\documentclass{amsart}
\usepackage{graphics,graphicx} 
\usepackage{fancybox}

\usepackage[T1]{fontenc}

\usepackage[cp1250]{inputenc}
\usepackage[arrow, matrix, curve]{xy}
\usepackage{amsmath,amsthm,amssymb}
\usepackage{amscd}
\xyoption{frame}

\newenvironment{Proof of}[1]{\emph{Proof of #1.}}{$\qquad \square$\par}

\DeclareMathOperator{\dashind}{-Ind}
\DeclareMathOperator{\Ind}{Ind}
\DeclareMathOperator{\Aut}{Aut}
\DeclareMathOperator{\Irr}{Irr}

\DeclareMathOperator{\clsp}{\overline{span}}

\newcommand{\K}{\mathcal K}

\newcommand{\RR}{\mathcal{R}}

\newcommand{\X}{\widehat{X}}

\newcommand{\Y}{\widehat Y}
\newcommand{\V}{\widetilde V}

\newcommand{\sal}{\widehat \alpha}

\newcommand{\LL}{\mathcal{L}}
\newcommand{\Qq}{\mathcal Q}
\newcommand{\al}{\alpha}

\newcommand{\FF}{\mathcal F}

\newcommand{\OO}{\mathcal O}
\newcommand{\NO}{\mathcal{NO}}

\newcommand{\B}{\mathcal B}

\newcommand{\SA}{\widehat{A}}
\newcommand{\SB}{\widehat{B}}

\newcommand{\G}{\mathcal G}
\newcommand{\C}{\mathbb C}

\newcommand{\Z}{\mathbb Z}
\newcommand{\N}{\mathbb N}
\newcommand{\T}{\mathbb T}

\newtheorem{thm}{Theorem}[section]\newtheorem{lem}[thm]{Lemma} 
\newtheorem{prop}[thm]{Proposition} 
\newtheorem{cor}[thm]{Corollary}
\theoremstyle{definition} 
\newtheorem{defn}[thm]{Definition}
\newtheorem{ex}[thm]{Example}
\newtheorem{rem}[thm]{Remark}

\title[Topological aperiodicity for product systems]{Topological aperiodicity for product systems over  
semigroups of Ore type}

\date{December 24, 2013}

\author{Bartosz Kosma  Kwa\'sniewski}
\address{ Institute of Mathematics, Polish Academy of Science,  ul. \'Sniadeckich 8, PL-00-956 Warszawa, Poland // Institute of Mathematics, University  of Bialystok,  ul. Akademicka 2, PL-15-267  Bialystok, Poland}
\email{bartoszk@math.uwb.edu.pl}  

\author{Wojciech Szyma\'nski}
\address{Department of Mathematics and Computer Science, The University of Southern Denmark, 
Campusvej 55, DK--5230 Odense M, Denmark}
\email{szymanski@imada.sdu.dk}

\keywords{$C^*$-correspondence, Hilbert bimodule, product system, Cuntz-Pimsner algebra, Ore semigroup, 
               topological freeness, Fell bundle,  topological graph, uniqueness theorem}

\subjclass[2010]{46L05}

\begin{document}
\begin{abstract}
We prove a version of uniqueness theorem for Cuntz-Pimsner algebras of discrete product systems over  semigroups of Ore type. To this end, we introduce  Doplicher-Roberts picture of Cuntz-Pimsner algebras, and the semigroup dual to a product system of 'regular' $C^*$-correspondences. Under a certain aperiodicity condition on the latter, we obtain  the uniqueness theorem and a   simplicity criterion for the algebras in question. These results generalize the corresponding ones for crossed products by discrete groups,   due to Archbold and Spielberg, and for Exel's crossed products, due 
to  Exel and Vershik. They also give interesting conditions for topological higher rank graphs and $P$-graphs, and apply to the new Cuntz $C^*$-algebra $\mathcal{Q}_\mathbb{N}$ arising from the "$ax+b$"-semigroup over $\mathbb{N}$.
\end{abstract}
\maketitle

\tableofcontents

\section{Introduction}

A fundamental problem in every theory dealing with  $C^*$-algebras generated by operators satisfying prescribed relations  is the uniqueness of such objects. More specifically, suppose $\RR$ is a set of $C^*$-algebraic relations on a set of generators $\G$, and suppose there is a mapping $\pi:\G\to \B(H)$  such that  $\{\pi(g)\}_{g\in \G}$ are non-zero bounded  operators on a Hilbert space $H$ which satisfy relations $\RR$. We call such $\pi$  faithful representation of $(\G,\RR)$, and 
we denote by $C^*(\pi)$ the $C^*$-algebra generated by $\{\pi(g)\}_{g\in \G}$.  The pair $(\G,\RR)$ has \emph{uniqueness property} if for any  two faithful representations $\pi_1$, $\pi_2$ of $(\G,\RR)$ the mapping
$$
\pi_1(g) \longmapsto \pi_2(g), \qquad g\in \G,
$$
extends to the (necessarily unique) isomorphism $C^*(\pi_1)\cong C^*(\pi_2)$. Results stating that a certain class of  relations possesses the above property are called \emph{uniqueness theorems}.  For reasonable pairs 
$(\G,\RR)$, see for instance  \cite{blackadar}, there exists a universal $C^*$-algebra $C^*(\G,\RR)$ for the above defined representations of $(\G,\RR)$. Clearly,  $(\G,\RR)$ has the uniqueness property if and only if $C^*(\G,\RR)$  exists and for any faithful representation $\pi$ of  $(\G,\RR)$ the natural epimorphism from $C^*(\G,\RR)$ onto $C^*(\pi)$ is actually an isomorphism. 

Among the oldest and best studied uniqueness theorems are those related to $C^*$-dynamical systems. Recall 
that such a system $(A, \alpha, G)$, consists of a $C^*$-algebra $A$ and a group action $\alpha:G\to \Aut(A)$. Uniqueness result in this context applies to the associated crossed product. Starting at least from the 
sixties, uniqueness theorems for crossed products began to appear in   connection with various problems such as properties of the Connes spectrum, proper outerness, ideal structure, or spectral analysis of functional-differential operators, see \cite[p. 225, 226]{AL}, \cite{Arch_Spiel} and \cite{kwa} for relevant surveys. One of the most popular conditions of this kind, known today as \emph{topological freeness}, was probably for the first time explicitly stated in \cite{OD} for $\Z$-actions. O'Donovan  proved in \cite{OD} that if the set of periodic points for the dual action $\widehat{\alpha}$ on the spectrum $\widehat{A}$ of $A$ has empty interior   then the crossed product $A\rtimes_\alpha \Z$ has intersection property, which is equivalent to the uniqueness property as defined above. This  result was generalized to the case of amenable discrete groups \cite{AL} and then to arbitrary discrete groups \cite{Arch_Spiel}. More specifically, by \cite[Theorem 1]{Arch_Spiel} topological freeness of $\widehat{\alpha}$ implies intersection property for $A\rtimes_\alpha G$, and this is equivalent to the uniqueness property if and only if action $\alpha$ is amenable in the sense that the full crossed product $A\rtimes_\alpha G$ and the reduced crossed product $A\rtimes_{\alpha,r} G$ are naturally isomorphic. This formulation is very convenient as it allows to investigate  amenability and topological freeness of $\alpha$ independently. Moreover, it can be used to study the structure of the reduced crossed product $A\rtimes_{\alpha,r} G$. 
We recall that for a separable  $A$ and $G=\Z$, or if $A$ is commutative and $G$ amenable discrete,  topological freeness of $\widehat{\alpha}$ is equivalent to the uniqueness property for $A\rtimes_\alpha G$, see 
\cite[Theorem 10.4]{OlPe} and \cite[Theorem 2]{Arch_Spiel}, respectively. However, it is known that 
already for $\Z^2$ actions topological freeness  is only sufficient but not necessary for the uniqueness property,  
\cite[Remark on page 123]{Arch_Spiel}. 

Another line of research leading towards numerous uniqueness theorems was  initiated by the seminal work of Cuntz and Krieger, \cite{CK1980}. In particular, \cite[Theorem 2.13]{CK1980} states that the Cuntz-Krieger relations possess the uniqueness property  if the underlying matrix $A$ satisfies condition (I).  Since then, similar results  concerning various generalizations of the algebra $\OO_A$ are usually called Cuntz-Krieger uniqueness theorems.   The diagram in Figure \ref{c-k theorems} presents  certain such theorems  relevant to the present paper; each item contains  the  name of  universal algebras, the condition which is (at present known to be) equivalent to uniqueness property for the corresponding defining relations, and the names of authors who introduced the condition. An arrow from $A$ to $B$ indicates that  algebras in question and the condition in $B$ can be viewed as generalizations of the ones in $A$. We provide more details and explanations in Section \ref{Applications and examples}.  

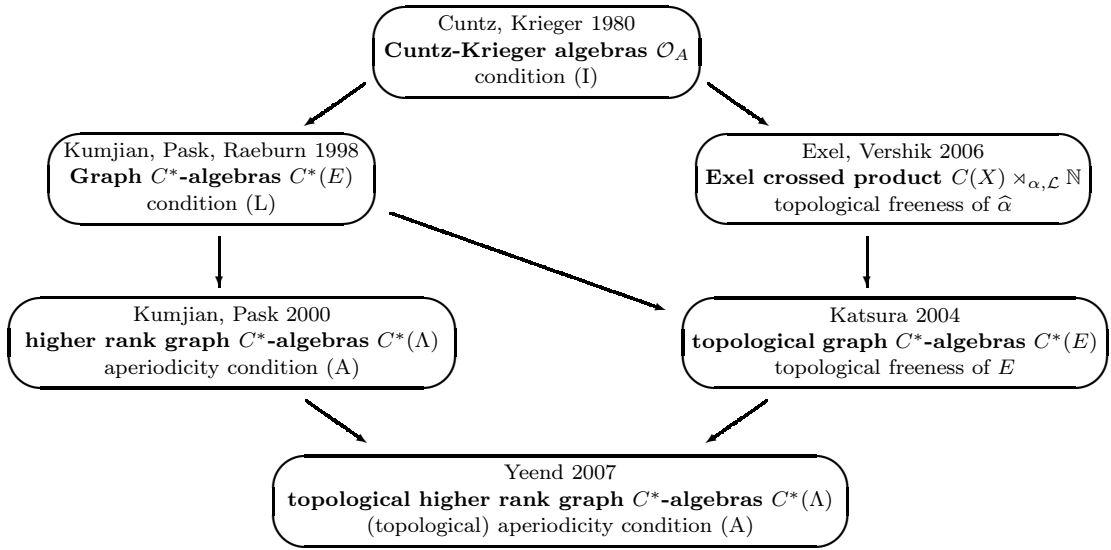
\begin{figure}[htb]
\begin{center} 
 \begin{picture}(350,195)(0,2)  
  \footnotesize
\cornersize{0.8}   \thinlines

\put(103,180){\Ovalbox{\begin{minipage}{4.1 cm}\begin{center}
Cuntz, Krieger 1980 \\
\textbf{Cuntz-Krieger algebras $\OO_A$} \\
condition (I)
\end{center}
\end{minipage}}} 
\linethickness{0.27mm}
\qbezier(100,170)(90,163)(80,156)
\put(80,156){\vector(-3,-2){3}}

\qbezier(230,170)(240,163)(250,156)
\put(250,156){\vector(3,-2){3}}
  
 \qbezier(45,112)(45,104)(45,96)
\put(45,96){\vector(0,-1){3}}
  
 \qbezier(290,112)(290,104)(290,96)
\put(290,96){\vector(0,-1){3}}

\qbezier(78,50)(88,43)(98,36)
\put(98,36){\vector(3,-2){3}}

\qbezier(232,36)(242,43)(252,50)
\put(232,36){\vector(-3,-2){3}}

\qbezier(110,121)(161,103)(212,85)
\put(212,85){\vector(3,-1){2}}

\put(-20,132){\Ovalbox{\begin{minipage}{4.1cm}\begin{center}
Kumjian, Pask, Raeburn 1998 \\
\textbf{Graph $C^*$-algebras $C^*(E)$}\\
condition (L)
\end{center}
\end{minipage}}} 

\put(225,132){\Ovalbox{\begin{minipage}{5cm}\begin{center}
Exel, Vershik 2006 \\
\textbf{Exel crossed product $C(X)\rtimes_{\alpha, \mathcal{L}}\N$} \\
topological freeness of $\widehat{\alpha}$
\end{center}
\end{minipage}}} 

\put(-35,70){\Ovalbox{\begin{minipage}{5.7 cm}\begin{center}
Kumjian, Pask 2000 \\
\textbf{higher rank graph $C^*$-algebras $C^*(\Lambda)$}\\
aperiodicity condition (A)
\end{center}
\end{minipage}}} 

\put(220,70){\Ovalbox{\begin{minipage}{5.4cm}\begin{center}
Katsura 2004 \\
\textbf{topological graph $C^*$-algebras $C^*(E)$} \\
topological freeness of $E$
\end{center}
\end{minipage}}} 

\put(65,10){\Ovalbox{\begin{minipage}{7.4 cm}\begin{center}
Yeend 2007 \\
\textbf{topological higher rank graph $C^*$-algebras $C^*(\Lambda)$} \\
(topological) aperiodicity condition (A)
\end{center}
\end{minipage}}} 
\end{picture} 
\end{center}
\caption{Cuntz-Krieger uniqueness theorems\label{c-k theorems}}
\end{figure}

The $C^*$-algebras associated with topological graphs were introduced in \cite{ka1} as a generalization of both graph $C^*$-algebras and crossed products of commutative $C^*$-algebras by $\Z$-actions. Similarly, $C^*$-algebras arising from topological higher rank graphs \cite{yeend} include as examples  crossed products of commutative $C^*$-algebras by $\Z^k$-actions. Algebras associated to topological higher rank graphs provide 
interesting examples of a  general, intensively investigated but still largely undeveloped theory of algebras associated with product systems over semigroups, \cite{CLSV}. One of the main aims of the present article is 
initialization of a systematic and unified approach to the study of uniqueness properties for universal 
$C^*$-algebras $C^*(\G,\RR)$. To this end, we establish certain  general results for Cuntz-Pimsner algebras associated with product systems over a large class of semigroups, and with coefficients in an arbitrary (not necessarily commutative) $C^*$-algebra $A$.

Uniqueness theorems  are often studied via the associated gauge  action of a dual group $\widehat{G}$ or a coaction of a relevant group $G$, e.g. see  \cite{KLQ}, \cite{CLSV}. In general, existence of  such an additional structure on a universal $C^*$-algebra $C^*(\G,\RR)$ can be thought of as arising from a symmetry in relations $\RR$.  It establishes a Fell bundle structure $\{B_t\}_{t\in G}$ on $C^*(\G,\RR)$. If $G=\Z$ and the Fell bundle $\{B_k\}_{k\in \Z}$ is semisaturated then $C^*$-algebra $C^*(\G,\RR)$ is naturally isomorphic to the crossed product $B_0 \rtimes_{B_1} \Z$, \cite{AEE}, where $B_1$ is treated as a Hilbert bimodule over $B_0$. The first named author proved in \cite{kwa} a uniqueness theorem for $B_0 \rtimes_{B_1} \Z$ under the assumption that a partial homeomorphism of $\widehat{B}_0$ given by Rieffel's induced representation functor $B_1-\Ind$ is topologically free. It seems plausible that similar techniques may lead to a generalization of \cite[Theorem 2.2]{kwa} 
to Fell bundles over arbitrary discrete groups. However, in many important cases (e.g. those listed in Figure 1 above) 
the initial data correspond to semigroups rather than groups. 
The analysis in \cite{kwa-interact} shows that in the context of Cuntz-Pimsner algebras associated with product systems over semigroups $P$, passing from the initial algebra $A$ to the core $B_0$ is  a very nontrivial procedure even in the case $A\cong \C^n$ and $P=\N$. That is why we pursue here a  more ambitious program  
focused on semigroups rather than groups. 

Our initial object is  a product system of $C^*$-correspondences $X$ over a discrete semigroup $P$ and  
with coefficients  in an arbitrary $C^*$-algebra $A$, as defined in \cite{F99}. We impose two critical restrictions 
on the product systems in question, one on the underlying semigroup $P$ and one on the structure of 
fibers $X_p$, $p\in P$. Namely, we assume that $P$ is an Ore semigroup. (Actually, we consider slightly more general semigroups,  satisfying only one-sided cancellation, see Subsection 2.5 below.)  
Such semigroups arise naturally in many contexts,  including dilations, \cite{Laca}, interactions, \cite{exel4}, 
and skew rings, \cite{AGBGP}. Among examples one finds all groups and all commutative cancellative semigroups. 
About the fibers $X_p$, $p\in P$, we assume that the left action of $A$ is given by an injective homomorphism 
into the compacts ${\mathcal K}(X_p)$. We call such an $X$ \emph{regular product system}. 

In the present paper, we are primarily focused on investigations of the Cuntz-Pimsner algebra $\OO_X$ 
associated to a regular product system $X$, as in \cite{F99}. Under our assumptions on $X$ and $P$, 
Fowler's definition seems to work particularly well. For instance, when $P$ is a positive cone in an ordered 
quasi-lattice group $(G,P)$, then $\OO_X$ coincides with the Cuntz-Nica-Pimsner algebra $\mathcal{NO}_X$,  \cite{F99}, \cite{SY},  \cite{CLSV}. However, we stress that the very definition of the Cuntz-Nica-Pimsner 
algebra $\mathcal{NO}_X$ puts severe restrictions on the class of semigroups $P$ one may consider. 
In particular, $P$ itself cannot be a group and this excludes many interesting examples. By contrast, the algebra 
$\OO_X$ does not have this drawback and our results reinforce the perception that (under our assumptions) it is the right object to study. 

In Section \ref{regularproductsytems}, we analyze the structure of the algebra $\OO_X$ associated to a regular 
product system $X$. We show (see Theorem \ref{structure theorem} below) that $\OO_X$ can be 
constructed in the spirit of the Doplicher-Roberts algebras arising in the abstract 
duality theory for compact groups, \cite{dr}. More precisely, we show that  $X$ gives rise to a right tensor
$C^*$-precategory $\K_X$ over the semigroup $P$,  cf. \cite{dr}, \cite{kwa-doplicher}, and $\OO_X$ is a completion of a graded $*$-algebra whose fibers are direct limits of elements of  $\K_X$. In the case $P=\N$ 
(and with no further assumptions on $X$) such an approach was elaborated in  \cite{kwa-doplicher}. 
This description immediately implies that the universal representation of $X$ in $\OO_X$ is injective, thus 
answering a question going back to Fowler's original paper \cite[Remark 2.10]{F99}. 
It also allows us to view and study $\OO_X$ as a cross sectional algebra of a certain Fell bundle $\{(\OO_X)_g\}_{g\in G(P)}$ over the enveloping group $G(P)$ of $P$. Taking advantage of this picture, 
we define the reduced Cuntz-Pimsner algebra $\OO_X^r$ of $X$ as the reduced cross sectional algebra of  $\{(\OO_X)_g\}_{g\in G(P)}$ \cite{Exel}, \cite{qui:discrete coactions}. In the case $\OO_X=\mathcal{NO}_X$, 
our $\OO_X^r$ coincides with the co-universal algebra $\mathcal{NO}_X^r$ defined in \cite{CLSV}. 

In Section \ref{Dual objects}, we present a novel construction of a semigroup  $\X$ dual  to a regular  product system $X$. Elements of $\X$ are multivalued maps on the spectrum of the coefficient algebra $A$. When $P=G$ is a group, these maps are honest homeomorphisms arising through Rieffel's induction, cf. \cite{kwa}. 
The semigroup $\X$ is particularly well suited for the study of uniqueness property and related questions. 

In Section \ref{A uniqueness theorem and simplicity criteria for Cuntz-Pimsner algebras}, 
we formulate a topological aperiodicity condition in terms of the semigroup $\X$.  This is the key ingredient 
entering  our {\em uniqueness theorem}, see Theorem \ref{Cuntz-Krieger uniqueness theorem} below. 
We prove that if $\X$ is topologically aperiodic, then for any faithful Cuntz-Pimsner representation $\psi$ of $X$ 
there exists a conditional expectation from the $C^*$-algebra generated by $\psi(X)$ onto its core 
$C^*$-subalgebra.  Such  conditional expectations  are main tools in analysis of representations and ideal 
structure of $C^*$-algebras under consideration. In particular, they are of critical importance in 
various gauge-invariant uniqueness theorems, see, for instance,  \cite{KLQ}, \cite{CLSV}, 
\cite[Chapter 3]{Raeburn}. When $\OO_X=\OO_X^r$, our uniqueness theorem  states that a representation of $\OO_X$ is faithful if and only if it is faithful on the algebra of coefficients $A$. As a corollary to 
Theorem \ref{Cuntz-Krieger uniqueness theorem}, we obtain the following simplicity criterion. If $\X$ is 
topologically aperiodic then $\OO_X^r$ is simple if and only if $X$ is minimal, see Theorem \ref{simplicity} below.

Applications and examples of our main results are presented in Section \ref{Applications and examples}. 
Logical relationships between the  topological aperiodicity of $\X$ and other aperiodicity conditions mentioned above, when applied to particular examples, are presented schematically on Figure \ref{aperiodicity conditions}. More specifically, in Subsection \ref{Product systems of Hilbert bimodules, Fell bundles and dual partial group 
actions} we consider product systems $X$ whose fibers are Hilbert bimodules. We show that  under 
this assumption the semigroup $\X=\{\X_p\}_{p\in P}$ consists of  partial homeomorphisms and generates 
a partial action of $G(P)$ on $\widehat{A}$, see Proposition \ref{partial dynamical system out of the 
product system} below. In this setting, topological freeness implies topological aperiodicity. 
As a bonus, we obtain uniqueness theorems and simplicity criteria for cross sectional algebras of 
saturated Fell bundles (Corollary \ref{uniqueness theorem and simplicity criterion for cross-sectional algebras}) 
and for twisted crossed products by semigroups of injective endomorphisms with hereditary ranges (Proposition \ref{partial dynamical system out of semigroup of endomorphisms}). 

\begin{figure}[htb]
\begin{center} \begin{picture}(240,146)(0,50)
\small
 \put(56,120){\Ovalbox{\begin{minipage}{4cm}\begin{center}
topological aperiodicity \\
for product systems 
\end{center}
\end{minipage}} 
}
\put(107,167){$\xymatrix{ \,\,   \ar@{=>}[d]   \\     \,\,\,  }$ } 
 \put(50,182){\framebox{$\begin{array}{c} \text{topological freeness} \\  \text{for groups of automorphisms}\end{array}$}}

\put(185,120){$\Longleftrightarrow$ }   \put(207,120){\framebox{$\begin{array}{c} \text{topological freeness} \\  \text{for covering maps}\end{array}$}}

\put(107,106){$\xymatrix{ \,\,   \ar@{=>}[d]   \\     \,\,\,  }$ }  \put(40,59){\framebox{$\begin{array}{c} \text{aperiodicity condition} \\  \text{for topological higher rank graphs}
\end{array}$}}

 \put(30,120){$\Longleftarrow$ }   \put(-45,120){\framebox{$\begin{array}{c}\text{simplicity of} \\  \text{Cuntz's }\mathcal{Q}_\N\end{array}$}}

  \end{picture} \end{center}
  \caption{Relationship between aperiodicity conditions\label{aperiodicity conditions}}
 \end{figure}
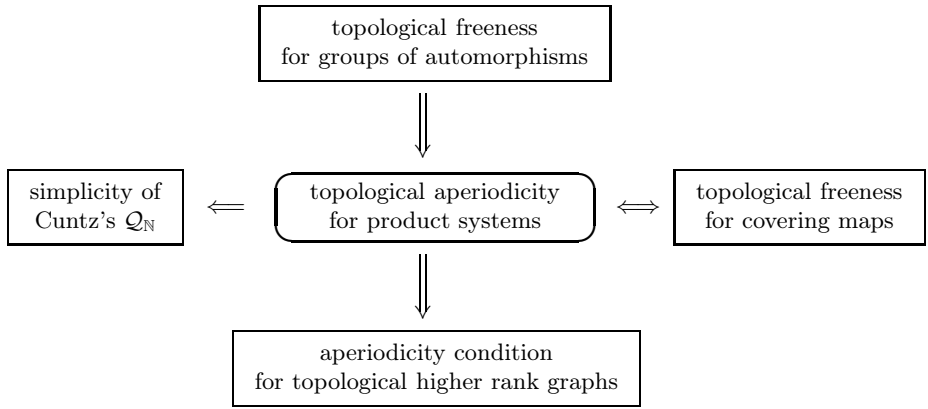

Another motivation for our work comes from theory of graph algebras and their generalizations. Any topological graph $E$ gives rise to a product system $X$ over the semigroup of natural numbers, \cite{ka1}. 
In this context, topological aperiodicity of $\X$ turns out to be strictly stronger than  topological freeness of $E$, but these notions coincide when the range map of $E$ is injective. For example, this latter condition holds 
for topological graphs arising from Exel's crossed products by covering maps, \cite{exel_vershik}, \cite{brv}. 
Our results  give necessary and sufficient conditions for uniqueness and simplicity of such crossed products, 
see Example \ref{Exel-Vershik} below. In Subsection \ref{product system of topological graphs over P}, 
we look at    topological higher rank graphs, \cite{yeend}, and their corresponding product systems over 
$P=\N^k$, \cite{CLSV}. In fact,  we consider certain topological $P$-graphs where $P$ is an arbitrary semigroup of Ore type  and thus we obtain  completely new  uniqueness and simplicity conditions for the associated $C^*$-algebras (discrete $P$-graphs where $(G,P)$ is a quasi-lattice ordered group were considered in \cite{Raeburn_Sims}, \cite{BSV}). As a final example we show that our results yield a quick and elegant way to  see simplicity of the Cuntz algebra $\Qq_\N$, \cite{Cun1}, which has a nice representation as $\OO_X$ where $X$ is a natural product system  as described in \cite{HLS}, see subsection \ref{Cuntz subsection}. 

Finally, we would like to point out two additional applications of our general structural result for $\OO_X$, 
Theorem \ref{structure theorem}. Firstly, we use it to reveal group grading and establish  non-degeneracy 
of the twisted crossed product by a semigroup action of injective endomorphisms, see Proposition 
\ref{pomocnicze on crossed products} below. Secondly, we give a natural definition of the  $C^*$-algebra $C^*(\Lambda,d)$  and the reduced $C^*$-algebra $C^*_r(\Lambda,d)$ associated to a product system of topological graphs over $P$, \cite{fs},  see Subsection \ref{product system of topological graphs over P}.  
These constructions generalize $C^*$-algebras associated to topological higher rank graphs and discrete $P$-graphs \cite{BSV}. Significantly,   
the Cuntz algebra $\Qq_\N$ can be modeled as a  $C^*$-algebra $C^*(\Lambda,d)$ associated to 
a topological $P$-graph   $(\Lambda,d)$ where $P=\N^*$, see Remark \ref{last remark} below.

\subsection{Acknowledgements}

The first  named author was partially supported by the NCN  Grant number DEC-2011/01/B/ST1/03838. The second named author was supported by the FNU Project Grant `Operator algebras, dynamical systems and quantum information theory' (2013--2015). This research was supported by a Marie Curie Intra European Fellowship within the 7th
European Community Framework Programme.

\section{Preliminaries}

This section contains the necessary preliminaries. In addition to more standard material, we discuss multivalued 
maps in Subsection \ref{multi maps section} and semigroups of Ore type in Subsection \ref{Ore subsection}. 


\subsection{Multivalued maps}\label{multi maps section}

We follow  standard conventions, cf. for instance \cite[Chapter 5]{rockafellar}, apart from notion of continuity which will not play any important role in the sequel. Let $M$ and $N$ be sets and $2^N$ be the family of all subsets of $N$. A {\em multivalued mapping} from $M$ to $N$ is by definition  a mapping from $M$ to $2^N$. We denote such a multivalued mapping $f$ by 
$f:M\to N$. Also, we identify the usual (single-valued) mappings with multivalued mappings taking values in singletons. We denote 
$$
D(f):=\{x\in M: f(x)\neq \emptyset\}, \qquad f(M):=\{y\in N: y \in f(x) \textrm{ for some } x\in M\}=\bigcup_{x\in M} f(x)
$$
the domain and the image of $f$ respectively. We put $f(A):=\bigcup_{x\in A} f(x)$ for a subset 
$A$ of $M$, and   define preimage of $B\subseteq  N$ to be the set
$$ 
f^{-1}(B):=\{x\in M: f(x)\cap B \neq \emptyset \}.
$$ 
This goes perfectly well with the natural definition of the multivalued inverse $f^{-1}$ of $f$, where
 $$
y\in  f^{-1}(x) \stackrel{def}{\Longleftrightarrow}  x \in f(y).
 $$
 For two multivalued  mappings $f,g:M\to N$  we write $f\subset g$ whenever $f(x)\subset g(x)$ for all $x\in M$. 
  Composition  of two multivalued maps $f:M\to N$ and $g:N\to L$ is the multivalued map $g\circ f:M\to L$ given by  
  $$
 (g\circ f)(x):=\bigcup_{y \in f(x)} g(y).
$$ 
One checks that the obvious rule $(f\circ g)^{-1}=g^{-1}\circ f^{-1}$ holds. However note that
\begin{equation}\label{multivalued identity}
(f\circ f^{-1})(x)=\bigcup_{ x\in f(y)} f(y) 
\end{equation}
is either empty or it is a subset containing $x$, possibly larger than $\{x\}$. The former happens when 
$x$ does not belong to the range of $f$ and the latter otherwise. 

If $M$ and $N$ are topological spaces,  we say that a multivalued map $f:M\to N$ is continuous if 
$f^{-1}(U)$ is open for every open subset $U$ of $N$. In the literature, this is usually  
taken as a definition of  lower  semi-continuity. But since we will not make use of upper semi-continuity we do not make a distinction.  


\subsection{Hilbert modules, $C^*$-correspondences and induced representations}

Throughout  this section, $A$, $B$ and $D$ are $C^*$-algebras. We adhere to the convention that $\beta(A,B)=\clsp\{\beta(a,b)\in D\mid a\in A,b\in B\}$ 
for  maps $\beta\colon A\times B\to D$  such as inner products, multiplications or representations. By  homomorphism, epimorphism, etc. we always mean an involution preserving map.  All ideals in 
$C^*$-algebras are  assumed to be closed and two-sided. 

We  adopt the standard notations and definitions of objects related to  Hilbert modules, cf. for instance \cite{morita}. A right Hilbert $B$-module is a Banach space $X$ which is a right $B$-module equipped with an  
$B$-valued inner product $\langle \cdot , \cdot \rangle _B:X\times X \to B$.  If  $X$,  $Y$ are right Hilbert 
$B$-modules then  $\LL(X,Y)$ stands for the space of adjointable operators from $X$ into $Y$. Also,  
the  space  of "compact" operators from $X$ to $Y$ is defined as 
$$
\K(X,Y)=\clsp\{\Theta_{y,x}: x\in X,y\in Y\}\subseteq \LL(X,Y), 
$$ 
where 
$$
\Theta_{y,x}(z)=y\langle x,z\rangle_B, \;\;\; z\in X.
$$ 
In particular,  $\K(X):=\K(X,X)$ is 
an ideal in the $C^*$-algebra $\LL(X):=\LL(X,X)$. 

A \emph{$C^*$-correspondence} from $A$ to $B$ is a right Hilbert $B$-module $X$ equipped with 
a homomorphism $\phi_X:A\rightarrow\LL(X)$.  We refer to $\phi_X$  as to the left action of $A$
on $X$ and write $a\cdot x=\phi_X(a)x$, for $a\in A$, $x\in X$. If $A=B$ then we call $X$  a 
$C^*$-correspondence with coefficients in $A$. A \emph{Hilbert $A$-$B$-bimodule} is a $C^*$-correspondence 
$X$ from $A$ to $B$ equipped with  a left $A$-valued inner product
${_A\langle} \cdot , \cdot \rangle:X\times X \to A$ such that 
$$
x\langle y , z \rangle_B={_A\langle} x , y \rangle z, \qquad x,y,z\in X.
 $$ 
Equivalently, $X$ is both a left Hilbert $A$-module and a right Hilbert $B$-module satisfying the above condition. 
If, in addition,  ${_A\langle} X , X \rangle=A$ and $\langle X , X \rangle_B=B$, then $X$ is an imprimitivity 
$A$-$B$-bimodule. For instance, every $C^*$-algebra $A$ can be considered a $C^*$-correspondence 
(actually, an imprimitivity $A$-$A$-bimodule), denoted $_AA_A$, 
where $\langle a, b\rangle _A=a^*b$, $_A\langle a, b \rangle = ab^*$, and both left and right action is simply 
multiplication in $A$. 

We note that there is a one-to-one correspondence between representations $\pi:A\to \B(H)$ 
of $A$ on a Hilbert space $H$ and $C^*$-correspondences $X=H$ from $A$ to $\C$ (where left action 
is induced by $\pi$). We say that such $C^*$-correspondences associated to the representation $\pi$. 
Furthermore, any right Hilbert $A$-module can be considered  a Hilbert 
$\K(X)$-$A$-bimodule, where ${_{\K(X)}\langle} x, y\rangle=\Theta_{x,y}$.

If $X$ is a right Hilbert $A$-module and $Y$ is a Hilbert $A$-$C$-bimodule, then the internal tensor 
product $X\otimes_A Y=\clsp\{x\otimes_A y: x\in X, y\in Y\}$ (balanced over $A$) is a right Hilbert 
$C$-module  with the right action induced from $Y$ and the $C$-valued inner product given by 
$$
\langle x_1\otimes_A y_1 , x_2\otimes_A y_2\rangle_C=\langle y_1, \phi_Y(\langle x_1,
x_2\rangle_A) y_2\rangle_C,\,\,\, \textrm{ for }x_i\in X \textrm{ and } y_i\in Y,\,\, i=1,2.
$$  
If, in addition,  $X$ is a $C^*$-correspondence from $B$ to $A$, then $X\otimes_A Y$ is a 
$C^*$-correspondence from $B$ to $C$ with the left action implemented by the
homomorphism $B\ni a \mapsto\phi_X(a)\otimes_A 1_Y\in\LL(X\otimes Y)$, where $1_Y$ is the unit in $\LL(Y)$. In the sequel, in order not to overload notation, we will often write simply $X\otimes Y$ and $x\otimes y$ for tensor products, when $A$ is understood. 

In the above scheme, a particularly important special case  occurs when $Y$ is a $C^*$-correspondence  
from $A$ to ${\mathbb C}$ associated  to a representation $\pi:A\to \B(H)$. Then for   any 
$C^*$-correspondence $X$ from $B$ to $A$ the $C^*$-correspondence $X\otimes_A Y$ is  associated 
to a certain representation of $B$ which we denote by $X\dashind(\pi)$ and call  \emph{representation induced from $\pi$ by $X$}. More precisely,  let $X\otimes_\pi H = \clsp X\otimes H$ be a Hilbert space equipped 
with the inner product
$$
\langle x_1\otimes_\pi h_1, x_2\otimes_\pi h_2 \rangle_{\C} = \langle h_1,\pi(\langle x_1, x_2 \rangle_{A})h_2\rangle_{\C}.
$$
Then $X\dashind(\pi)$ is a representation of $B$ on $X\otimes_\pi H$ such that 
\begin{equation}\label{induced representation definition}
X\dashind(\pi)(b)  (x\otimes_\pi h) = (b x)\otimes_\pi h, \qquad  b\in B.
\end{equation}
In particular, if  $X$ is  an imprimitivity $B$-$A$-bimodule, then by the celebrated Rieffel's result, cf. e.g.  
\cite[Theorem 3.29, Corollaries 3.32 and 3.33]{morita},  the induced representation functor $X\dashind$ 
factors through 
to the homeomorphism $[X\dashind]:\SA\to \widehat{B}$ between the spectra of $A$ and $B$. 
The inverse of this homeomorphism is given by induction with respect to a Hilbert module dual to $X$.
Here, a dual to  a right Hilbert $A$-module $X$ means a left Hilbert $A$-module $\widetilde{X} $ for which 
there exists an antiunitary $\flat:X \to \widetilde{X}$.  A natural model for  $\widetilde{X}$ is $\K(X,{_A}A_A)$ where $\flat(x)y=\langle x, y\rangle_A$. In particular, if $X$ is a Hilbert $A$-$B$-bimodule 
then $\widetilde{X}$ is a Hilbert $B$-$A$-bimodule.


\subsection{Product systems, their representations and Cuntz-Pimsner algebras}\label{Product systems preliminaries}

Let $A$ be a $C^*$-algebra and $P$ a discrete semigroup with identity $e$. 
A \emph{product system} over $P$ with coefficients in $A$ is a semigroup $X= \bigsqcup_{p\in P}X_{p}$, 
equipped with a semigroup homomorphism $d\colon X \to P$ such that
\begin{enumerate}\renewcommand{\theenumi}{P\arabic{enumi}}
\item $X_p = d^{-1}(p)$ is a $C^*$-correspondence with coefficients in  $A$ for each $p\in P$.
\item $X_e$ is the standard bimodule $_AA_A$.
\item The multiplication on $X$ extends to isomorphisms $X_p \otimes_A X_q \cong X_{pq}$ 
for $p,q \in P \setminus \{e\}$ and the right and left actions of $X_e = A$ on each $X_p$.
\end{enumerate}
For each $p\in P$, we denote by $\langle\cdot,\cdot\rangle_p$ the $A$-valued
inner product on $X_p$  and by $\phi_p$ the homomorphism from $A$ into $\LL(X_p)$
 which implements the left action of $A$ on $X_p$.
  Given $p, q \in P$ with $p \not= e$, there is a homomorphism $\iota^{pq}_p \colon 
\LL(X_p) \to \LL(X_{pq})$ characterised by
\begin{equation}
\iota^{pq}_p(T)(xy) = (Tx)y,\,\,\,  \text{ where $x \in X_p$, $y \in X_{q}$ and $T \in \LL(X_p)$.}
\end{equation}
We recall   that the  map 
\begin{equation}\label{C-correspondence isomorphism}
X_p\ni x\to t_x \in \K(A,X_p) \qquad \textrm{ where } t_x(a)=xa,
\end{equation} yields a $C^*$-correspondence isomorphism  $X_p\cong \K(A,X_p)$. Here 
$\K(A,X_p)$ is a $C^*$-correspondence with $A$-valued inner product $\langle T,S\rangle_A=T^*S$ 
and point-wise  actions. Thus we may define $\iota^p_e \colon \K(X_e)\to
\LL(X_{p})$ simply by letting $\iota^p_e(t_a)=\phi_p(a)$ for  $p\in P$, $a\in A$, \cite[\S  2.2]{SY}.

 A map $\psi$ from $X$ to a  $C^*$-algebra $B$ is a Toeplitz
\emph{representation of $X$ in $B$} if the following conditions hold:
\begin{enumerate}\renewcommand{\theenumi}{T\arabic{enumi}}
\item  for each $p\in P\setminus\{e\}$, $\psi_p:=\psi\vert_{X_p}$ is linear,  and  $\psi_e$ is a homomorphism, 
\item $\psi_p(x)\psi_q(y)=\psi_{pq}(xy)\; \; $ for $ \; x\in X_p$, $y\in X_q$, $p,q\in P$,
\item $\psi_p(x)^*\psi_p(y)=\psi_e(\langle x, y\rangle_p)$ for $x, y\in X_p$.
\end{enumerate}
It is well known that, for each $p\in P$ there exists a $*$-homomorphism
$\psi^{(p)} : \K(X_p) \longrightarrow B$ such that $\psi^{(p)}(\Theta_{x,y})=
\psi_p(x)\psi_p(y)^*\, , \; \text{for}\; x,y\in X_p$. The representation $\psi$ is 
called {\it Cuntz-Pimsner covariant}  if
\begin{itemize}
\item[(CP)]  $\psi^{(p)}(\phi_p(a))=\psi_e(a)$ for all $a\in A$ and $p\in P$.
\end{itemize}

As introduced by Fowler \cite{F99}, the Cuntz-Pimsner algebra $\OO_X$ of a product system 
$X$ is a universal $C^*$-algebra
for the Cuntz-Pimsner covariant representations. We denote by $j_X$ the universal representation of $X$ in $\OO_X$. Hence for any Cuntz-Pimsner representation $\psi:X\to B$ there is a unique epimorphism
$\Pi_\psi:\OO_X\to C^*(\psi(X))$ such that $j_X(x)=\psi(x)$ for all $x\in X$. We call $\Pi_\psi$ `the 
integrated representation'. 

It is well known and not hard  to see that a necessary condition for  $j_X$ to be injective (and hence for $\OO_X$ to be nondegenerate) is that all of the homomorphisms $\phi_p$, $p\in P$, are injective. It is known that this condition is also  sufficient \cite[Corollary 5.2]{SY} when  $P$ is a directed positive cone in a  quasi-lattice ordered group $(G, P)$ and each $\phi_p$ acts by compacts. In this  case, $\OO_X$ coincides with the so-called 
Cuntz-Nica-Pimsner algebra $\NO_X$ introduced in \cite{SY}.     
We recall that a partially ordered group  
$(G,P)$, consisting of a group $G$ and its subsemigroup $P\subseteq G$ such that  
$P \cap P^{-1} = \{e\}$,   is a \emph{quasi-lattice ordered group} if, under the partial order 
$g \le h \iff g^{-1}h \in P$, any two elements $p, q$ in $G$ with a common
upper bound in $P$ have a least common upper bound $p \vee q$
in $P$, \cite{N} and \cite[Lemma 7]{CL2002}. 
The semigroup $P$ is \emph{directed} if each pair of elements in $P$ has an upper bound:
$$
(\forall {p,q \in P}) \,\, (\exists {s \in P})\,\,\,     p,q \leq s. 
$$ 


\subsection{Co-actions, Fell bundles and their $C^*$-algebras}

Let $G$ be a discrete group. The shortest definition of a \emph{Fell bundle} (also called $C^*$-algebraic bundle) over $G$ is that it is a collection $\B=\{B_g\}_{g\in G}$ of closed subspaces of a $C^*$-algebra $B$ such that $B_g^*=B_{g^{-1}}$ and $B_g B_h\subseteq B_{gh}$ for all $g,h \in G$. Then the direct sum 
$\bigoplus_{g\in G} B_g$ is a $*$-algebra. In general, there are many different $C^*$-norms on 
$\bigoplus_{g\in G} B_g$. However, it is well known that there always exists a maximal such norm 
and it satisfies, cf. \cite[Lemma 1.3]{qui:discrete coactions} or \cite{Exel}, the inequality 
\begin{equation}\label{topological grading inequality}
\|a_e\|\leq \|\sum_{g \in G} a_g\|, \qquad \textrm{ for all } \sum_{g \in G} a_g\in \bigoplus_{g\in G} B_g,\,\,a_g\in B_g, \,\, g\in G.
\end{equation}
The completion of $\bigoplus_{g\in G} B_g$ in  this maximal $C^*$-norm is called \emph{cross sectional algebra}
 of $\B$ and it is denoted $C^*(\B)$. Moreover, it follows from 
\cite[Theorem 3.3]{Exel} that there is also  a minimal $C^*$-norm on $\bigoplus_{g\in G} B_g$ satisfying \eqref{topological grading inequality} and a completion of $\bigoplus_{g\in G} B_g$ in this minimal $C^*$-norm is naturally isomorphic to the \emph{reduced cross sectional algebra} $C_r^*(\B)$, as introduced in 
\cite[Definition 2.3]{Exel} or \cite[Definition 3.5]{qui:discrete coactions}. Both algebras 
$C^*(\B)$ and $C^*_r(\B)$  are equipped with natural coactions of $G$. 

We recall (see, for example, \cite{qui:discrete coactions}) that a  coaction
of a discrete group $G$ on a $C^*$-algebra $B$ is an injective and nondegenerate homomorphism
$\delta \colon B\to B\otimes C^*(G)$  satisfying  the coaction identity 
$(\delta\otimes id_{C^*(G)})\circ \delta= (id_B \otimes \delta_G)\circ \delta$, where   
$\delta_G \colon C^*(G)\to C^*(G) \otimes C^*(G)$ is  given by
$\delta_G(g)=i_G(g)\otimes i_G(g)$ and $i_G:G \to M(C^*(G))$ is the universal representation of $G$. 
The spectral subspaces $B_g^\delta:=\{a\in B\mid \delta(a)=a\otimes i_G(g)\}$, $g\in G$, 
form a Fell bundle $\B=\{B_g^\delta\}_{g\in G}$ and yield a $G$-gradation of $B$ such that    $B=\overline{\bigoplus_{g\in G}B_g}$. Moreover,  the norm on $\bigoplus_{g\in G}B_g$ inherited from 
$B$ satisfies inequality \eqref{topological grading inequality}.


\subsection{Semigroups of Ore type}\label{Ore subsection}

A  (left-reversible) Ore semigroup is a cancellative semigroup $P$ which is \emph{left reversible}, that is $sP \cap tP\neq \emptyset$, for all $s,t \in P$. Usually one considers right-reversible Ore semigroups  but the left version is more appealing for our purposes. It is well known that a semigroup $P$ is Ore precisely when it can be 
embedded in a group $G$ in such a way that  $G=PP^{-1}$, cf. \cite{CP,Laca,AGBGP}. For further reference, 
we include an elementary proof of a slightly more general statement. 

We let $P$ be a left reversible and left cancellative semigroup with identity $e$. We call such a $P$  
\emph{semigroup of Ore type} (it is Ore if and only if it is right cancellative as well). The semigroup 
structure induces a left-invariant preorder on $P$ defined  as:
\begin{equation}\label{semigroup pre-order}
p\leq q\; \stackrel{def}{\Longleftrightarrow}\;  pr=q \,\,\,\,\textrm{ for some } r\in P.
\end{equation}
If $(G, P)$ is ordered group the (pre)order on $P$  coincides with the one inherited from $(G, P)$. In terms of preorder \eqref{semigroup pre-order}, left reversibility of $P$ simply means that  $P$ is directed.

If  $p,q\in P$  then left cancellativity implies that relation $pr=q$  determines $r\in P$ uniquely.  
Thus, we introduce the notation
 $$
 p^{-1}q: =r \quad \textrm{ whenever }\,\,\,  pr=q.
 $$
The  \emph{enveloping group} or a \emph{group of fractions}   of $P$ is the universal group with 
the set of generators equal to $P$  and relations $xy=z$, whenever such identity holds in $P$. 
To construct the group of fractions explicitly, we first introduce a relation $\sim$  on $P\times P$ as:
\begin{equation}\label{equivalence ore relation}
(p_1,p_2)\sim (q_1,q_2)\,\,  \stackrel{def}{\Longleftrightarrow} \,\, p_1 p=q_1 q, \quad p_2p=q_2q\,\,\, \textrm{ for some } p,q\in P.
\end{equation}
\begin{lem} Relation \eqref{equivalence ore relation} is an equivalence relation on $P\times P$.
\end{lem}
\begin{proof}
Reflexivity and symmetry are obvious. To show transitivity, assume that  in addition to \eqref{equivalence ore relation} we also have $(q_1,q_2)\sim (r_1,r_2)$, where $q_1 s=r_1 r$, $q_2s=r_2r$ for some $s,r\in P$. 
Then for any $t\geq q, s$ we have
$$
p_1p (q^{-1}t)=q_1q(q^{-1}t)=q_1 t= q_1 s (s^{-1}t)= r_1r (s^{-1}t).  
$$ 
Similarly, one shows that $p_2p (q^{-1}t)=r_1 r (s^{-1}t)$. Hence $(p_1,p_2)\sim (r_1,r_2)$.
\end{proof}
We use square brackets to denote the equivalence classes of relation \eqref{equivalence ore relation}:
$$
[p_1,p_2]:=\{(q_1,q_2)\in P\times P: (p_1,p_2)\sim (q_1,q_2)\}, 
$$
and denote the quotient set by $G(P):=P\times P/\sim$. We define a product on $G(P)$  by the formula
\begin{equation}\label{group action}
[p_1,p_2] \circ [q_1,q_2] := [p_1 (p_2^{-1}s),q_2(q_1^{-1}s)] \,\,\,\textrm{ for some } \,\,\, s\geq p_2, q_1 .
\end{equation}
This definition is correct due to left cancellativity of $P$. 
\begin{prop}[Ore's Theorem]
For the left cancellative and directed  semigroup $P$ the quotient set $G(P)$ with the product 
\eqref{group action} is a group such that $G(P)=\iota(P)\iota(P)^{-1}$, where 
$$
P \ni p \stackrel{\iota}{\longmapsto} [p,e]\in G(P)
$$
is a semigroup homomorphism. This homomorphism is injective if and only if $P$ is right cancellative. 
\end{prop}
\begin{proof}
Clearly, $[e,e]$ is a neutral element for product $\circ$ and $[q,p]$ is the inverse of $[p,q]$. To show associativity 
of $\circ$, let $p_i,q_i,r_i \in P$, $i=1,2$, and choose any   $s\geq p_2, q_1$ and $t\geq q_2(q_1^{-1}s) , r_1$. Then 
$$
t= q_2(q_1^{-1}s)z\,\,\,\textrm{  and  }\,\,\, s=p_2y \,\,\,\textrm{  for some  } \,\,\, y,z\in P.
$$
 Thus we have
\begin{align*}
\big([p_1,p_2]\circ [q_1,q_2]\big)\circ [r_1,r_2] & = [p_1 (p_2^{-1}s),q_2(q_1^{-1}s)] \circ [r_1,r_2]
\\
&= [p_1 (p_2^{-1}s) \Big(\big(q_2(q_1^{-1}s)\big)^{-1}t \Big) , r_2(r_1^{-1}t)]
\\
&= [p_1 (p_2^{-1}s)z) , r_2(r_1^{-1}t)].
\end{align*}
On the other hand,  putting $u:=t$  and 
$$
w:=q_1(q_2^{-1}u)=q_1(q_2^{-1}t)=q_1(q_2^{-1}q_2(q_1^{-1}s))z=q_1(q_1^{-1}s)z=sz=p_2yz
$$
we get $u \geq q_2 , r_1$ and  $w \geq q_1(q_2^{-1}u), p_2$. 
Consequently,  
\begin{align*}
[p_1,p_2]\circ \big([q_1,q_2]\circ [r_1,r_2]\big) & =  [p_1,p_2]\circ  [q_1 (q_2^{-1}u),r_2(r_1^{-1}u)] 
\\
&= [p_1 (p_2^{-1}w), r_2(r_1^{-1}u) \big(q_2(q_1^{-1}u)\big)^{-1}w]
\\
&= [p_1 (p_2^{-1}s)z, r_2(r_1^{-1}t)],
\end{align*}
which proves associativity of $\circ$. As $[p,e] \circ [q,e] =[pq,e]$, because $q\geq q,e$, 
we see that $\iota$ is a semigroup homomorphism. Moreover,  $
[p,e]=[q,e]$ if and only if   $p t=q t$ for some $t\in P$, and therefore $\iota$ is injective if and 
only if $P$ is right cancellative.
\end{proof}
\begin{rem}\label{kernel of Ore}
It follows from the above that the relation $p \sim_R q$ ${\Longleftrightarrow}$ $pr=qr$, for some $r\in P$, is a semigroup congruence on $P$ and the quotient semigroup $P/\sim_R$ is an Ore semigroup whose 
enveloping group is naturally isomorphic to $G(P)$. 
\end{rem}


\section{Regular product systems of  $C^*$-correspondences and their $C^*$-algebras}
\label{regularproductsytems}  

In this section, we first introduce and discuss certain product systems of $C^*$-correspondences  satisfying additional regularity conditions,  and then construct their associated Cuntz-Pimsner algebras and their reduced versions in the spirit of the Doplicher-Roberts algebras \cite{dr}. Our construction involves an object that may be viewed as a right tensor $C^*$-precategory over $P$, see   \cite{kwa-doplicher}. Regular product systems introduced in this section and their $C^*$-algebras will play a central role in the remainder of this article. 


\subsection{Regular product systems and their right tensor $C^*$-precategories}

\begin{defn}
Let $X$ be a $C^*$-correspondence with coefficients in $A$. We  say  $X$ is  \emph{regular} if its left action 
is injective and  via compact operators, that is  
\begin{equation}\label{standing assumptions}
 \ker\phi=\{0\}\qquad \textrm{ and  }\qquad  \phi(A)\subseteq \K(X).
 \end{equation}
  We say that a product system  $X := \bigsqcup_{p\in P}X_{p}$ over a  semigroup $P$ is  {\em regular} 
if each fiber $X_p$, $p\in P$, is a regular $C^*$-correspondence.
 \end{defn}

The notions of  regularity and tensor product are compatible  in the sense that the tensor product of two regular $C^*$-correspondences is automatically regular, see Proposition \ref{regular proposition} below.  

Before proceeding further we need a technical Lemma \ref{lemma on tensoring regular correspondences} 
whose assertion  is probably well known to experts, but  we include a  proof for the sake of completeness. 
\begin{lem}\label{lemma on tensoring regular correspondences}
Let $Y$ be a regular $C^*$-correspondence with coefficients in $A$ and let  $X$, 
$Z$ be right Hilbert $A$-modules.
\begin{itemize}
\item[i)] For each $x\in X$, the mapping 
$$
Y\ni y\stackrel{T_x}{\longrightarrow} x\otimes y\in X\otimes Y
$$ 
is compact, that is $T_x\in \K(Y,X\otimes Y)$. Furthermore, we have $\|T_x\|=\|x\|$.
\item[ii)] For each $S\in \K(X,Z)$, we have $S\otimes 1_Y \in \K(X\otimes Y,Z\otimes Y)$ and the mapping
\begin{equation}\label{tensoring on the right}
\K(X,Z)\ni S \longmapsto S\otimes 1_Y \in \K(X\otimes Y,Z\otimes Y)
\end{equation}
is isometric. It is surjective whenever $\phi_Y:A\to \K(Y)$ is.
\end{itemize}
\end{lem}
\begin{proof}
Ad (i). Note that  $T_x\in \LL(Y,X\otimes Y)$ and $T_x^*( x_0\otimes y_0)= \langle x, x_0\rangle_A y_0$. Let $x=x_0a$ for some $x_0\in X$ and $a\in A$. Then $\phi_Y(a)=\lim_{n\to \infty } \sum_{i}\Theta_{\eta_i^n,\mu_i^n}$ for some $\eta_i^n$, $\mu_i^n \in Y$. Thus 
$$
T_x=T_{x_0}\phi_Y(a)=\lim_{n\to \infty } \sum_{i}T_{x_0}\Theta_{\eta_i^n,\mu_i^n}=\lim_{n\to \infty } \sum_{i}\Theta_{x_0\otimes \eta_i^n,\mu_i^n}\in \K(Y,X\otimes Y).
$$
As $\phi_Y$ is isometric,  $\phi_Y(\langle x, x\rangle_A)=S^*S$ for some 
$S\in \LL(Y)$ with $\|S\|=\|x\|$, and hence
$$
\|T_x\|^2=\sup_{y\in Y, \|y\|=1} \|\langle y, \phi_Y(\langle x,x\rangle_A)  y\rangle_A\| 
=\sup_{y\in Y, \|y\|=1} \|\langle Sy, S y\rangle_A\|=\|S^2\|=\|x\|^2.
$$
Ad (ii). Let  $x\in X$, $z\in Z$ and consider $T_x \in \K(Y,X\otimes Y)$ and 
$T_z\in \K(Y,Z\otimes Y)$ as in item i). Since
$$
\Theta_{z,x}\otimes 1_Y= T_zT_x^*\in \K(X\otimes Y,Z\otimes Y),
$$
we have $\K(X,Z)\otimes 1_Y \subseteq \K(X\otimes Y,Z\otimes Y)$. To show that mapping 
\eqref{tensoring on the right} is isometric, we first consider the case $Z=X$. Then \eqref{tensoring on the right} 
is a homomorphism of $C^*$-algebras and therefore it suffices to show it is injective. To this end, let $S\in \K(X)$ 
be non-zero. Take $x\in X$ such that $S  x\neq 0$ and  $y\in Y$ such that 
$\phi_Y(\langle  Sx, Sx\rangle_A)y \neq 0$. Then 
$$
\left\langle ( S\otimes 1_Y) x\otimes y ,  Sx\otimes \phi_Y(\langle  Sx, Sx\rangle_A)y \right\rangle = \left\langle \phi_Y(\langle  Sx, Sx\rangle_A)y, \phi_Y(\langle  Sx, Sx\rangle_A)y \right\rangle_A \neq 0,
$$
which implies $ S\otimes 1_Y\neq 0$. Consequently, $\|S\otimes 1_Y\|=\|S\|$. Now getting back to the general case (when $Z$ is arbitrary), for $S\in \K(X,Z)$ we have 
$$
\|S\otimes 1_Y\|^2=\|S^*S\otimes 1_Y\|= \|S^*S\|=\|S\|^2. 
$$
If the homomorphism $\phi_Y:A\to \K(Y)$ is surjective, then it is an isomorphism and  simple computations show that for $x \in X$, $y_1,y_2\in Y$ and $z\in Z$ we have
$$
\Theta_{z\otimes y_1,x\otimes y_2}= \Theta_{z\phi_Y^{-1}(\Theta_{y_1,y_2}),x}\otimes 1_Y.
$$ 
This implies that mapping \eqref{tensoring on the right} is surjective.
\end{proof}

\begin{prop}\label{regular proposition} 
Tensor product of regular $C^*$-correspondences  is a regular $C^*$-correspondence. 
\end{prop} 
\begin{proof} 
If $X$ and $Y$ are $C^*$-correspondences over $A$ then  the left action of $A$ on $X\otimes Y$ is $\phi_{X\otimes Y}=\phi_X\otimes 1_Y$. Hence if $X$ and $Y$ are regular, then $\phi_{X\otimes Y}$ is injective and acts by compacts, by Lemma \ref{lemma on tensoring regular correspondences} part (ii).  
\end{proof} 

Now, let $X$ be  a regular product system  over $P$.  The family 
$$
\K_X:=\{\K(X_{q},X_{p})\}_{p,q\in P}
$$ 
forms  in a natural manner a $C^*$-precategory,  \cite[Definition 2.2]{kwa-doplicher}. We will  describe a right tensoring structure on $\K_X$  by introducing a family of mappings $\iota_{{p},{q}}^{{pr},{qr}}: \K(X_{q},X_{p}) \to \K(X_{qr},X_{pr})$, $p,q,r\in P$, cf. \cite[Example 3.2]{kwa-doplicher}, which extends the standard family of diagonal homomorphisms $\iota_{q}^{qp}$ defined in Subsection \ref{Product systems preliminaries} (when restricted to compact operators). If $q\neq e$ we put 
$$
\iota^{pr,qr}_{p,q}(T)(xy):=(Tx)y,\,\, \,\,\,\,\,\textrm{ where } x\in X_q,\, y\in X_r  \textrm{ and } T\in \K(X_{q},X_{p}).
$$
Note that under the canonical isomorphism $ X_{pq}\cong X_p \otimes_A X_q$ operator  
$\iota^{pr,qr}_{p,q}(T)$ corresponds to $T\otimes 1_{X_r}$. Hence by part (ii) of 
Lemma \ref{lemma on tensoring regular correspondences}, $\iota^{pr,qr}_{p,q}(T)\in \K(X_{qr},X_{pr})$ 
and  $\iota^{pr,qr}_{p,q}$ is isometric. Similarly, in the case $q=e$, 
using \eqref{C-correspondence isomorphism},  the formula 
$$
\iota^{pr,r}_{p,e}(t_x)(y):=xy,\,\, \,\,\,\,\,\textrm{ where }\, y\in X_r  \textrm{ and } t_x\in \K(X_{e},X_{p}),  x\in X_p,
$$
yields a well defined map. By Lemma \ref{lemma on tensoring regular correspondences} part (i), this is an isometry from $\K(X_{e},X_{p})$ into $\K(X_{r},X_{pr})$. Note that  $\iota_{p,p}^{pr,pr}=\iota_{p}^{pr}$. 

\begin{defn}\label{right tensor C-precategory}
The $C^*$-precategory $\K_X:=\{\K(X_{q},X_{p})\}_{p,q\in P}$ equipped with the family of maps 
$\{\iota_{{p},{q}}^{{pr},{qr}}\}_{p,q,r\in P}$  defined above is called 
\emph{a right tensor $C^*$-precategory associated to the regular product system} $X$. 
\end{defn} 

 \begin{rem} $\K_X$ is a $C^*$-precategory in the sense of \cite[Definition 2.2]{kwa-doplicher} whose  
objects are elements of $P$.  One readily sees that the isometric linear maps $\iota_{{p},{q}}^{{pr},{qr}}: \K(X_{q},X_{p}) \to \K(X_{qr},X_{pr})$, $p,q,r\in P$, satisfy 
\begin{equation}\label{star and mulitpilcativity}
 \iota_{{p},{q}}^{{pr},{qr}}(T)^*=\iota_{{q},{p}}^{{qr},{pr}}(T^*),\qquad  \iota_{{p},{q}}^{{pr},{qr}}(T) \iota_{{q},{s}}^{{qr},{sr}}(S)= \iota_{{p},{s}}^{{pr},{sr}}(TS),
 \end{equation}
 \begin{equation}\label{semigroup property}
  \iota_{{pr},{qr}}^{{prs},{qrs}}(\iota_{{p},{q}}^{{pr},{qr}}(T)) = \iota_{{p},{q}}^{{prs},{qrs}}(T),
 \end{equation}
for all $T \in \K(X_{q},X_{p})$, $S \in \K(X_{s},X_{q})$, $p,q,r,s\in P$. Thus, if we adopt the notation
 $$
T\otimes 1_r:=\iota^{pr,qr}_{p,q}(T),\qquad \quad T\in\K(X_{q},X_{p}),\,\, p,q \in P,
 $$
then \eqref{star and mulitpilcativity} means that $\otimes 1_r: \K_X\to \K_X$ is a $C^*$-precategory monomorphism sending $p$ to $pr$,  see \cite[Definition 2.8]{kwa-doplicher}, and \eqref{semigroup property} states that $\otimes 1_r\circ \otimes 1_s=\otimes 1_{rs}$, that is $\{\otimes 1_r\}_{r\in P}$ is a semigroup action on $\K_X$.  In particular, the pair $(\K_X, \{\otimes 1_r\}_{r\in P})$, which is another presentation of $(\K_X,  \{\iota_{{p},{q}}^{{pr},{qr}}\}_{p,q,r\in P})$, is a (strict) right tensor $C^*$-category (cf. e.g., \cite{dr}) 
when each of the algebra $\K(X_p)$, $p\in P$, is unital.
 \end{rem}
The following lemma could be considered a  counterpart of \cite[Proposition 3.14]{kwa-doplicher}. 
 
\begin{lem}\label{going forward lemma} 
Let  $\psi$ be  a  representation of a regular product system $X$ over a semigroup $P$ in a $C^*$-algebra $B$. 
For each $p,q\in P$ we have a  contractive  linear map $\psi_{p,q} : \K(X_q,X_p) \longrightarrow B$  
determined by the formula 
\begin{equation}\label{right tensor representation}
\psi_{p,q}(\Theta_{x,y})=
\psi_p(x)\psi_q(y)^*\,\,\,  \; \text{for}\; x\in X_p,\,\, y \in X_q.
\end{equation}
Mappings  $\{\psi_{p,q}\}_{p,q\in P}$ satisfy
\begin{equation}\label{right tensor representation propert}
\psi_{p,q}(S)\psi_{q,r}(T)=\psi_{p,r} (ST)\,\,\,  \; \text{for}\; S\in \K(X_q,X_p),\,\, T \in \K(X_r,X_q),\,\, p,q,r \in P,
\end{equation}
and are all isometric if  $\psi$ is injective. 
If $\psi$ is Cuntz-Pimsner covariant, then 
\begin{equation}\label{right tensor representation property}
 \psi_{p,q}(S)=\psi_{pr,qr}(\iota_{p,q}^{pr,qr}(S)) \qquad  \textrm{ for all } p, q, r \in P \textrm{  and  } S\in \K(X_q,X_p). 
\end{equation}
\end{lem}
\begin{proof}
 It is not completely trivial but quite  well known that  \eqref{right tensor representation} defines a linear contraction which is isometric if $\psi_e$ is injective, see for instance the proof of Lemma 2.2 in \cite{KPW}. 
One readily sees that \eqref{right tensor representation propert} holds for 'rank one' operators 
$S=\Theta_{u,w}$, $T=\Theta_{v,z}$, and thus it  holds in general. To see \eqref{right tensor representation property}, suppose that $\psi$ is  a Cuntz-Pimsner covariant representation on a Hilbert space $H$ and $C^*(\psi(X))H=H$. Since  $\psi:X\to\B(H)$ is a semigroup homomorphism, the essential spaces 
$$
H_p:=\psi^{(p)}(\K(X_p))H=\psi_p(X_p)H
$$
of algebras $\psi^{(p)}(\K(X_p))$, $p\in P$,   form a decreasing family with respect to pre-order 
\eqref{semigroup pre-order}:
$$
p \leq q\,\, \Longrightarrow \,\, H_p \supseteq H_q.
$$
In particular,  $H=H_e=\psi_e(A)H$ and  actually $H=H_p$ for all $p\in P$, since 
$\psi_{e}(A)\subseteq \psi^{(p)}(\K(X_p))$ by  Cuntz-Pimsner covariance.
 Hence the linear span of elements of the form $\psi_{qr}(x_0y_0)h$, $x_0\in X_q$, $y_0\in X_r$, $h\in H$,  
is dense in $H$ and \eqref{right tensor representation property}  follows from the following computation:  
\begin{align*}
\psi_{pr,qr}(\iota_{p,q}^{pr,qr}(\Theta_{x,y}))\psi_{qr}(x_0y_0)&= \psi_{pr}(\iota_{p,q}^{pr,qr}(\Theta_{x,y})x_0y_0)=\psi_{pr}((\Theta_{x,y}x_0)y_0)
\\
&
=\psi_{pr}(x \langle y,x_0\rangle y_0)=\psi_{p}(x)  \psi_{q}(y)^* \psi_{q}(x_0) \psi_{r}(y_0)
\\
&
=\psi_{p,q}(\Theta_{x,y})\psi_{qr}(x_0y_0).
\end{align*}
 \end{proof}


\subsection{Doplicher-Roberts picture of a Cuntz-Pimsner algebra and its reduced version}

Throughout this subsection we assume that $X$ is a regular product system over a semigroup of Ore type, 
see Subsection \ref{Ore subsection}.
For the proof of the main result of this section we need  the following lemma, cf. 
\cite[Proposition 5.10]{F99}.

\begin{lem}\label{going forward lemma2} 
Suppose $\psi$ is a Cuntz-Pimsner covariant representation of  a regular product system $X$ 
over a semigroup $P$ of Ore type. 
\begin{itemize}
\item[i)] For all $x\in X_p$, $y\in X_q$ and $s\geq p,q$ we have
$$
\psi_p(x)^* \psi_q(y)\in \clsp\{ \psi(f)\psi(h)^*:  f\in X_{p^{-1}s}, h\in X_{q^{-1}s}\}.
$$ 
\item[ii)] We have the equality
\begin{align*}
&\clsp\{\psi(x)\psi^*(y): x,y\in X,\,\,  [d(x),d(y)]=[p,q]\}
\\
&=\clsp\{\psi(x)\psi^*(y): x\in X_{pr},y\in X_{qr},\,\,  r \in P\}.
\end{align*}
\item[iii)] $C^*(\psi(X))=\clsp\{ \psi(x)\psi(y)^*:  x,y\in X\}$. Furthermore, using the mappings defined in Lemma \ref{going forward lemma}  one can  arrange a dense subspace of $C^*(\psi(X))$ consisting of  
elements of the form
\begin{equation}\label{general form3} 
\psi^{(q)}(S_{q}) + \sum_{p\in F} \psi_{p,q}(S_{p,q})
\end{equation}
where $q\in P$ and $F\subseteq P$ is a finite set such that $q \nsim_R p$ for all $p\in F$, cf. 
Remark \ref{kernel of Ore}. 
\end{itemize}
\end{lem}
\begin{proof}
Ad (i). Write $x=Sx'$ with $S\in \K(X_{p})$ and  $x'\in X_{p}$,  and similarly $y=Ty'$ with $T\in \K(X_{q})$,  
$y'\in X_{q}$. Then using \eqref{right tensor representation}, \eqref{right tensor representation property} 
and \eqref{star and mulitpilcativity} we get
$$
 \psi_{p}(x)^*\psi_{q}(y)= \psi_{p}(x')^*\psi^{(p)}(S^*) \psi^{(q)}(T)\psi_q(y') = \psi_{p}(x')^*\psi^{(s)}(\iota_{p}^{s}(S^*) \iota_p^s(T))\psi_q(y').
$$  
Since $\iota_{p}^{s}(S^*) \iota_p^s(T) \in \K(X_s)$ we may approximate $\psi^{(s)}(\iota_{p}^{s}(S^*) \iota_p^s(T))$  with finite sums of operators of the form $\psi_{s}(f'f)\psi_{s
}(h'h)^*$, where $f'\in X_p$, $f\in X_{p^{-1}s}$ and   $h'\in X_q$, $h\in X_{q^{-1}s}$. Hence $\psi_{p}(x)^*\psi_{q}(y)$ can be approximated by finite sums of elements of the form 
$$
\psi_{p}(x')^*\psi_{s}(f'f)\psi_{s
}(h'h)^*\psi_q(y')=\psi_{p^{-1}s}(\langle x', f'\rangle_p f) \psi_{q^{-1}s} (\langle y', h'\rangle h)^*.
$$
This proves claim (i).
\\
Ad (ii). Clearly,  $\clsp\{\psi(x)\psi^*(y): x,y\in X,\,\,  [d(x),d(y)]=[p,q]\}$ contains $\clsp\{\psi(x)\psi^*(y): x\in X_{pr},y\in X_{qr},\,\,  r \in P\}$. To see the converse inclusion, we use the mappings introduced in Lemma \ref{going forward lemma} and assume that  $[p',q']=[p,q]$, that is $p'r'=pr$ and $q'r'=qr$ for some $r,r'\in P$. Then  by \eqref{right tensor representation property} for $T\in \K(X_{q'},X_{p'})$ we have
$$
\psi_{p',q'}(T)=\psi_{p'r',q'r'}(\iota_{p',q'}^{p'r',q'r'}(T))=\psi_{pr,qr}(\iota_{p',q'}^{p'r',q'r'}(T))\in \clsp\{\psi(x)\psi^*(y): x\in X_{pr},y\in X_{qr}\},
$$
which proves our claim.
\\
Ad (iii). Part (i) implies that $C^*(\psi(X))$ is the closure of elements of the form 
\begin{equation}\label{general form}
\sum_{i=1}^n \psi_{p_i}(x_i)\psi_{q_i}(y_i)^*,
\end{equation}
 where $p_i,q_i\in P$, $x_i \in X_{p_i}$, $y_i\in X_{q_i}$, $i=1,...,n $. Moreover, taking any $q_0\in P$ that dominates all $q_i$, $i=1,...,n$, and   writing $y_i=y_i'a_i$ with $y_i'\in X_{q_i}$, $a_i\in A$,  we get 
 $$
 \psi_{p_i}(x_i)\psi_{q_i}(y_i)^*= \psi_{p_i}(x_i)\psi^{(q_i^{-1}q_0)}(\phi_{q_i^{-1}q_0}(a_i^*))\psi_{q_i}(y_i')^*, \qquad i=1,...,n.
 $$
 Approximating $\psi^{(q_i^{-1}q_0)}(\phi_{q_i^{-1}q_0}(a_i^*))$ by finite sums of elements of the form  $\psi_{q_i^{-1}q_0}(u_i)\psi_{q_i^{-1}q_0}(v_i)^*$ we see that $\psi_{p_i}(x_i)\psi_{q_i}(y_i)^*$ can be  approximated by finite sums of elements of the form
 $$
 \psi_{p_i}(x_i)\psi_{q_i^{-1}q_0}(u_i)\psi_{q_i^{-1}q_0}(v_i)^*\psi_{q_i}(y_i')^*=
\psi_{p_iq_i^{-1}q_0}(x_iu_i) \psi_{q_0}(y_i'v_i)^*.
 $$ 
Thus  we see that the element \eqref{general form} can be presented in the form 
\begin{equation}\label{general form2}
\sum_{p\in F'} \psi_{p,q_0}(S_{p,q_0 })
\end{equation}
where $F'=\{p_iq_i^{-1}q_0: i=1,...,n\}\subseteq P$ is a finite set. Let $F_0=\{p\in F': q_0 \sim_R p\}$ and for each $p \in F_0$ choose $r_p\in P$ such that $pr_p=q_0r_p$. Let $r\in P$ be such that  $r\geq r_p$ for all 
$p \in F_0$, and put 
$$
q:=q_0r\quad \textrm{ and } \quad F:=\{pr: p \in F'\setminus F_0\}.
$$
Then  $pr=q$ for all $p \in F_0$, and $p\nsim_R q$ for all $p \in F$. By \eqref{right tensor representation propert}  we have $\psi_{p,q_0}(S_{p,q_0 })\in \psi_{pr,q_0r}(\K(X_{q_0r}, X_{pr}))=\psi_{pr,q}(\K(X_{q}, X_{pr}))$ and hence the element \eqref{general form2} can be presented in the form \eqref{general form3}.
\end{proof}

Now, we are ready to prove the main theorem of this section. It gives a direct construction 
of the Cuntz-Pimsner algebra $\OO_X$ of a regular product system $X$ as the full cross-sectional $C^*$-algebra 
of a suitable Fell bundle corresponding to the limits of directed systems of the compact operators arising 
from $X$. 


\begin{thm}\label{structure theorem} 
Let $X$ be a regular product system over a semigroup $P$ of Ore type and let  $G(P)$ be 
the enveloping group of $P$.  For each $[p,q]\in G(P)$ we define 
$$
B_{[p,q]}:=\underrightarrow{\,\lim\,\,} \K(X_{qr},X_{pr})
$$
to be the Banach space direct limit of the directed system $\left(\{\K(X_{qr},X_{pr})\}_{r\in P},\{\iota_{pr,qr}^{ps,qs}\}_{r,s\in P \atop r\leq s}\right)$. The family $\B=\{B_t\}_{t\in G(P)}$ is in a natural 
manner equipped with the structure of a Fell bundle over $G(P)$ and we have a canonical isomorphism
$$  
\OO_X\cong C^*(\{B_g\}_{g\in G(P)})
$$
from the Cuntz-Pimsner algebra $\OO_X$ onto the full cross-sectional $C^*$-algebra $C^*(\{B_g\}_{g\in G(P)})$.
In particular, 
\begin{itemize}
 \item[i)] the universal representation $j_X:X\to \OO_X$ is injective, 
\item[ii)] $\OO_X$ has a  natural grading $\{(\OO_X)_{g}\}_{g\in G(P)}$ over $G(P)$, such that 
\begin{equation}\label{Fell bundle fibres}
(\OO_X)_g=\clsp\{ j_X(x)j_X(y)^*:  x,y\in X, \,\, [d(x),d(y)]=g\}. 
\end{equation}
\item[iii)] for every injective representation $\psi$ of $X$, the integrated representation $\Pi_\psi$ of $\OO_X$ is isometric on each Banach space $(\OO_X)_g$, $g\in G(P)$, and thus it restricts to an isomorphism of the core $C^*$-subalgebra of $\OO_X$, namely
$$
(\OO_X)_e=\clsp\{ j_X(x)j_X(y)^*:  x,y\in X, \,\,  d(x)=d(y)\}.
$$ 
\end{itemize}
\end{thm}
\begin{proof} 
As the direct limit $\underrightarrow{\,\lim\,\,} \K(X_{qr},X_{pr})$ depends only on 'sufficiently large $r$', 
it follows immediately from \eqref{equivalence ore relation} that the limit does not depend on the choice 
of a representative of $[p,q]$ and thus $B_{[p,q]}$ is well defined. 
Let $\varphi_{p,q}:\K(X_{q},X_{p}) \to B_{[p,q]}$ denote the natural embedding of $\K(X_{q},X_{p})$ into  $B_{[p,q]}$. It is isometric because all the connecting maps $\iota_{{pr},{qr}}^{{ps},{qs}}$, $r\leq s$, are. 
Using the (inductive) properties of the mappings $\varphi_{p,q}$ and (right tensoring) properties \eqref{star and mulitpilcativity},  \eqref{semigroup property} of the  mappings  $\iota_{p,q}^{pr,qr}$, one sees that the formula
$$
\varphi_{p_1,p_2}(S) \circ \varphi_{q_1,q_2}(T):=  \varphi_{p_1 (p_2^{-1}s),q_2(q_1^{-1}s)}\Big(\iota_{p_1,p_2}^{p_1 (p_2^{-1}s),s}(S) 
\iota_{q_1,q_2}^{s,q_2(q_1^{-1}s)}(T)\Big),
$$ 
where $s\geq p_2, q_1$,  $S\in \K(X_{p_2},X_{p_1})$, $T\in \K(X_{q_2},X_{q_1})$, yields  
well defined bilinear  maps  
$$ 
\circ: B_{[p_1,p_2]} \times B_{[q_1,q_2]} \to B_{[p_1,p_2]\circ [q_1,q_2]}. 
$$ 
These maps establish an associative multiplication $\circ$ on $\{B_t\}_{t\in G(P)}$, satisfying 
$$
\|a\circ b\|\leq \|a\| \cdot \|b\|.
$$ 
Hence $\{B_t\}_{t\in G(P)}$ becomes a Banach algebraic bundle, cf. e.g. 
\cite[Definition 2.2. parts (i)--(iv)]{exel2}. Similarly,  formula
$$
\varphi_{p_1,p_2}(S)^*:= \varphi_{p_2,p_1}(S^*), \qquad S \in \K(X_{p_2},X_{p_1}), 
$$ 
defines  a '$*$' operation that satisfies  axioms \cite[Definition 2.2. parts (v)--(xi)]{exel2} and hence we 
get a Fell bundle structure on   $\{B_g\}_{g\in G(P)}$ (we omit straightforward but tedious verification 
of the details). 

Now, we view $C^*(\{B_g\}_{g\in G(P)})$ as a maximal $C^*$-completion of the direct sum 
$\bigoplus_{g \in G(P)} B_g$. 
Using the maps   \eqref{C-correspondence isomorphism},   we define mappings 
$$ 
\Psi:X=\bigsqcup_{p\in P}X_{p} \to C^*(\{B_g\}_{g\in G(P)})
$$ by 
\begin{equation}\label{representation from fibers}
X_p \ni x 
\longmapsto \varphi_{p,e}(t_x), \qquad p\in P.
\end{equation}
Since \eqref{C-correspondence isomorphism}  is  an isomorphism of $C^*$-correspondences, it follows  that $\Psi$  restricted to each summand $X_p$ is an injective representation of a $C^*$-correpondence. Moreover,  for $x\in X_p$, $y\in X_q$ we have $t_{xy}=i^{pq,q}_{p,e}(t_{x})t_y$ and thus
$$
\Psi(x)\Psi(y)=\varphi_{p,e}(t_{x})\circ \varphi_{q,e}(t_y)=\varphi_{pq,e}(i^{pq,q}_{p,e}(t_{x})t_y) =\varphi_{pq,e}(t_{xy})=\Psi(xy).
$$
Hence $\Psi$ is a faithful representation of the product system $X$ in $C^*(\{B_t\}_{t\in G(P)})$.  We recall that $\iota_{e,e}^{p,p}(t_a)=\iota_{e}^{p}(a)=\phi_p(a)$ and hence
$$
\Psi(a)=\varphi_{e,e}(t_a)=\varphi_{p,p}(\iota_{e,e}^{p,p}(t_a))=\varphi_{p,p}(\phi_p(a))=\Psi (\phi_p(a)), \qquad a\in A,\,\, p \in P,
$$
that is $\Psi$ is Cuntz-Pimsner covariant. Since $\Psi$ is injective, so is $j_X$ and claim (i) holds. 
Now, considering the integrated representation $\Pi_\Psi:\OO_X\to C^*(\{B_g\}_{g\in G(P)})$, 
for $x\in X_p$, $y\in X_q$ we have
\begin{equation}\label{extending to compacts}
\Pi_\Psi(j_X(x) j_X(y)^*)=\Psi(x)\circ \Psi(y)^*=\varphi_{p,e}(t_x)\circ \varphi_{e,q}(t_y^*)=\varphi_{p,q}(t_x t_y^*)=\varphi_{p,q}(\Theta_{x,y}). 
\end{equation}
It follows that $\Pi_\Psi$ maps 
$$
(\OO_X)_{[p,q]}:=\clsp\{ j_X(x)j_X(y)^*: x\in X_{pr},y\in X_{qr},\,\,  r \in P\}$$ 
onto $B_{[p,q]}$. Putting $g=[p,q]$ and using Lemma \ref{going forward lemma2} part (iii), 
we see that  $(\OO_X)_g$ is  given by \eqref{Fell bundle fibres}.
We claim that $\Pi_\Psi$ is injective on $(\OO_X)_g$. To see this, let $j_{p,q}$ denote the mappings from Lemma \ref{going forward lemma} associated to the universal representation $j_X$ and note that we have
$$
j_{ps,qs}\circ \iota_{pr,qr}^{ps,qs}=j_{pr,qr} \qquad \textrm{for } r\leq s
$$
by \eqref{right tensor representation property}. By the universal property of  inductive limits, 
there is a mapping 
$$
B_{[p,q]}\ni  \phi_{pr,qr}(T)\mapsto  j_{pr,qr}(T) \in (\OO_X)_{[p,q]},
$$
which is inverse to $\Pi_\Psi|_{[p,q]}$. Accordingly, $\Pi_\Psi$ is an epimorphism injective on each $(\OO_X)_g$. Since the spaces $B_g$, $g\in G(P)$, are linearly independent, so are $(\OO_X)_g$, $g\in G(P)$.  Consequently, in view of Lemma \ref{going forward lemma2} we have 
$$
\OO_X=\overline{ \bigoplus_{g\in G(P)} (\OO_X)_g}
$$
and claim (ii) follows. In particular, $\Pi_\Psi: \bigoplus_{g\in G(P)} (\OO_X)_g\to \bigoplus_{g\in G(P)} B_g$ is an isomorphism and as $C^*(\{B_g\}_{g\in G(P)})$ is the closure of $\bigoplus_{g \in G(P)} B_g$ in a maximal $C^*$-norm we see that $\Pi_\Psi$  actually yields the desired isomorphism $\OO_X\cong C^*(\{B_g\}_{g\in G(P)})$.

For the proof of part (iii), notice that we have just showed that $(\OO_X)_{[p,q]}$ is the closure of the 
increasing union $\bigcup_{r\in P} j_{pr,qr}(\K(X_{qr},X_{pr}))$, where 
$j_{pr,qr}:\K(X_{qr},X_{pr})\to (\OO_X)_{[p,q]}$ are  isometric maps. Similarly, if    $\psi$ is an injective 
covariant representation of $X$, then $\Pi_\psi((\OO_X)_{[p,q]})$ is the closure of the increasing 
union $\bigcup_{r\in P} \psi_{pr,qr}(\K(X_{qr},X_{pr}))$, and by Lemma \ref{going forward lemma}  
mappings $\psi_{pr,qr}:\K(X_q,X_p) \to \Pi_\Psi((\OO_X)_{[p,q]})$ are isometric. Since  
$\Pi_\psi\circ j_{pr,qr}=\psi_{pr,qr}$, $p,q,r\in P$,  it follows that surjection 
$\Pi_\psi:(\OO_X)_{[p,q]}\to \psi((\OO_X)_{[p,q]})$ is an isometry, since  it is isometric 
on a dense subset.
\end{proof}

\begin{rem}\label{remarkable remarks}
Theorem \ref{structure theorem} has a number of remarkable consequences. 

\smallskip\noindent (i)
The Cuntz-Pimsner algebra $\OO_X$  can be constructed  in a natural manner as the full cross-sectional 
algebra $C^*(\B)$ of the Fell bundle  $\B=\{B_t\}_{t\in G(P)}$. Thus it is justified to call the reduced 
cross-sectional algebra 
  $$
  \OO_X^r:=C^*_r(\B)
  $$ 
  the \emph{reduced Cuntz-Pimsner algebra of $X$}. In particular, $\OO_X^r$ is the $C^*$-algebra 
$C^*(j_X^r(X))$ generated by an injective Cuntz-Pimnser representation $j_X^r:X \to \OO_X^r=C^*_r(\{B_g\}_{g\in G(P)})$ acting according to \eqref{representation from fibers}. 
When  $P$ is Ore and $(G(P),\iota(P))$ is a quasi-lattice ordered group then  $\OO_X^r$ coincides with   
the co-universal $C^*$-algebra $\NO_X^r$ introduced and investigated in  \cite{CLSV}.

\smallskip\noindent (ii)
Our construction yields a faithful Cuntz-Pimsner representation of $X$ and thus the Cuntz-Pimsner 
algebra $\OO_X$ does not degenerate (it contains an isomorphic copy of $X$). This addresses the problem 
raised already by Fowler in \cite[Remark 2.10]{F99}. Until now, this problem was solved positively in the case $P$ is Ore and $(G(P),\iota(P))$ is a quasi-lattice ordered group, in which case $\OO_X$ coincides with the Cuntz-Nica-Pimsner algebra $\NO_X$.

\smallskip\noindent (iii)
When $P$ is Ore and $(G(P),\iota(P))$ is a quasi-lattice ordered group then part (iii) of Theorem 
\ref{structure theorem} coincides with \cite[Theorem 3.8]{CLSV}. In general, this result leads to (or actually could be considered as a version of) the so-called gauge invariant uniqueness theorem, cf. Proposition \ref{thm:projective property} below.
\end{rem}


\section{Dual objects}\label{Dual objects}

In essence, the dual objects we investigate are relations. However, we would like  to think of them in dynamical terms and therefore we will consider relations as multivalued maps, see subsection \ref{multi maps section} 
for the relevant terminology and conventions. 

 
\subsection{Multivalued maps dual  to homomorphisms of $C^*$-algebras}

Let $A$ be a $C^*$-algebra. We denote by $\simeq$ the unitary equivalence relation between 
representations of $A$, and by $[\pi]$ the corresponding equivalence class of $\pi:A\to \B(H)$. 
Spectrum $\SA=\{[\pi]: \pi\in\Irr(A)\}$ consists of the equivalence classes of all irreducible representations of 
$A$, equipped with the Jacobson topology. The relation $\leq $ of being a subrepresentation factors through $\simeq$ to a relation $\preceq$ on $\SA$. Namely, if $\pi:A\to \B(H_\pi)$ and $\rho:A\to \B(H_\rho)$ 
are representations of $A$, then
$$
 [\pi] \preceq [\rho] \Longleftrightarrow \exists\text{ isometry }U:H_\pi\to H_\rho  \,\, 
\text{ s. t. }(\forall{a\in A})\,\, \pi(a)=U^*\rho(a)U.
$$
Let $\al:A\to B$ be a homomorphism between two $C^*$-algebras. It is useful to think of the dual map 
we aim to define as a factorization of a multivalued map  $\sal_0:\Irr(B)\to\Irr(A)$ given by
\begin{equation}\label{pre dual}
\sal_0(\pi_B)=\{\pi_A\in \Irr(A): \pi_A \leq \pi_B\circ \al\}.
\end{equation} 
 The set  $[\sal_0(\pi_B)]:=\{[\pi_A]\in \SA: \pi_A \leq \pi_B\circ \al\}$ does not depend on the choice of a representative of the class $[\pi_B]$ and thus the following definition make sense.

\begin{defn} The \emph{dual  map} to a homomorphism $\al:A\to B$ is a multivalued map  
$\sal:\SB\to \SA$ given by the formula
\begin{align*}
\sal([\pi_B]):= & \{[\pi_A]\in \SA: [\pi_A] \preceq [\pi_B\circ \al]\} \\ 
= & \{[\pi_A]\in \SA: \pi_A \leq \pi_B\circ \al\}.
\end{align*}
\end{defn}

The range of $\sal$ behaves exactly as one would expect. But for non-liminal  $B$  the map $\sal$, and in 
particular its domain, has to be treated with care. Let us explain it with  help of  the following proposition 
and an example.

 \begin{prop}\label{something positive for Wojtek} 
For every homomorphism $\al:A\to B$ between 
two $C^*$-algebras, its image 
 $$
 \sal(\SB)=\{[\pi_A] \in \SA: \ker\pi_A \supseteq  \ker \al \}
 $$   is a closed subset of $\SA$. Its domain $D(\sal)$ is contained in an open subset 
$\{[\pi_B]\in \SB: \ker\pi_B \nsupseteq B\al(A)B\}$ of $\SB$.
 Moreover, if $B$ is liminal, then 
 $$
D(\sal)=\{[\pi_B]\in \SB: \ker\pi_B \nsupseteq B\al(A)B\}
 $$
 and $\sal:\SB\to \SA$ is continuous.
 \end{prop} 
 \begin{proof}
If $[\pi_A]\in \sal(\SB)$,  then  $\pi_A \leq \pi_B\circ \al$ for some  $\pi_B\in \Irr(B)$,  and hence 
$\ker\pi_A \supseteq  \ker \al$. Conversely, if $[\pi_A]\in \SA$ is such that $\ker\pi_A \supseteq  \ker \al$, then $\pi_A$ factors thorough to the irreducible representation   of $A/\ker\al \cong \al(A)$. Thus the formula $\pi(\al(a)):=\pi_A(a)$, $a\in A$, yields a well defined element of $\Irr(\al(A))$. Extending $\pi$  to any $\pi_B\in \Irr(B)$ one has $\pi_A\leq \pi_B\circ \al$.

Now, let $J$ be an ideal of $A$. Then $\widehat{J}=\{[\pi_A]\in \SA: \ker\pi \nsupseteq J\}$ 
is open and we have
\begin{align*}
[\pi_B]\in \sal^{-1}(\widehat{J}) & \Longleftrightarrow \exists_{\pi_A\in \Irr(A)}\,\, \pi_A\leq  \pi_B\circ \al,\,\,\, \ker\pi_A \nsupseteq J  
\\
&\Longrightarrow \ker(\pi_B\circ \al)\nsupseteq J
\\
&\Longleftrightarrow \ker\pi_B\nsupseteq \al(J)
\\
& \Longleftrightarrow \ker\pi_B\nsupseteq B\al(J)B. 
\end{align*}
That is, $\sal^{-1}(\widehat{J}) \subseteq \{\pi_B \in \SB:  \ker\pi_B\nsupseteq B\al(J)B\}$ and in particular  $D(\sal)=\sal^{-1}(\widehat{A})\subseteq \{\pi_B \in \SB:  \ker\pi_B\nsupseteq B\al(A)B\} $. 

If we additionally assume that $B$ is liminal, then for   $\pi_B \in \Irr(B)$ the  representation  $\pi_B\circ \al$ decomposes into a  direct sum of  irreducibles, see for instance \cite[\S 5.4.13]{Dix:C*-algebras}. 
Namely,  there is a subset $K$ of $\sal_0(\pi_B)$ such that $\pi_B\circ \al=\bigoplus_{\pi_A\in K}\pi_A\oplus 0$ (where $0$ stands for the zero representation and is  vacuous if $\pi_B\circ \al$ is nondegenerate).  Hence 
 the implication
$$
 \ker(\pi_B\circ \al)\nsupseteq J \Longrightarrow  
\exists{\pi_A\in K\subseteq \Irr(A)}\,\, \text{ s. t. }\pi_A\leq  \pi_B\circ \al,\,\,\, \ker\pi_A \nsupseteq J
$$
holds true. This combined with the preceding argument yields $\sal^{-1}(\widehat{J})=
\{\pi_B \in \SB:  \ker\pi_B\nsupseteq B\al(J)B\}$ and     the second part of the assertion follows. 
\end{proof} 

\begin{ex}
Let $H=L^2_\mu[0,1]$ with $\mu$ the Lebesgue measure. Put $B:=\B(H)$, $A:=L^{\infty}[0,1]$ and let $\al:A\to B$ be the monomorphism sending $a\in A$ to the operator of multiplication by $a$.
Then $\pi_B=id$ is irreducible and $\pi_B\circ \al$ is faithful but  $\sal([\pi_B])=\emptyset$. Accordingly,
$$
D(\sal)\neq \{[\pi_B]\in \SB: \ker\pi_B \nsupseteq B\al(A)B\}=\SB. 
$$
 \end{ex}


\subsection{Multivalued maps dual  to regular $C^*$-correspondences}

Let $X$ be a regular $C^*$-correspondence with coefficients in $A$. 
We may treat $X$   as a $\K(X)$-$\langle X,X\rangle _A$-imprimitivity bimodule and therefore the induced representation functor $X\dashind:\Irr(\langle X,X\rangle _A)\to \Irr(\K(X))$ factors through to  the homeomorphism $[X\dashind]: \widehat{\langle X,X\rangle}_A\to \widehat{\K(X)}$, which in turn may be 
viewed as a multivalued map  $[X\dashind]: \widehat{A}\to \widehat{\K(X)}$ with domain $D([X\dashind])=\widehat{\langle X,X\rangle}_A$. 

\begin{defn}\label{dual map def}
Let $X$ be a regular $C^*$-correspondence over $A$.  We define   \emph{dual  map} $\X:\SA\to \SA$ to 
$X$ as the following composition of multivalued maps
$$
\X=\widehat{\phi}\circ  [X\dashind ],
$$
where $\widehat{\phi}:\widehat{\K(X)}\to \SA$ is dual to the left action $\phi:A\to\K(X)$ of $A$ on $X$.

Alternatively, $\X$ is a factorization of the map $\X_0:=\widehat{\phi}_0\circ  X\dashind: \Irr(A)\to \Irr(A)$, 
cf. \eqref{pre dual}.
\end{defn}

\begin{prop}\label{basic properties of the dual maps}
The multivalued map dual to a regular $C^*$-correspondence $X$ is always surjective, that is  $\X(\SA)=\SA$.  
The domain of $\X$ satisfies the following inclusion
\begin{equation}\label{domain estimation}
D(\X)\subseteq \widehat{\langle  X,\phi(A)X\rangle _A}. 
\end{equation}
Note here that $\langle  X,\phi(A)X\rangle _A$ is an ideal in $A$. If, in addition, $A$ is liminal, then $\X$ is a continuous multivalued map and  we  have the equality in \eqref{domain estimation}; in particular,
 if  $X$ is  full and essential, then  $\X:\SA \to \SA$ is a continuous multivalued surjection with the full domain, $D(\X)=\SA$.

\end{prop}
\begin{proof}
As $[X\dashind]: \widehat{A}\to \widehat{\K(X)}$ is surjective and $\ker\phi=\{0\}$ we get $\X(\SA)=\SA$ by Proposition \ref{something positive for Wojtek}. Since $[X\dashind]: \widehat{\langle X,X\rangle _A}\to \widehat{\K(X)}$ is a homeomorphism, it follows from Proposition \ref{something positive for Wojtek}  that  
\begin{equation}\label{domain estimation2}
D(\X)\subseteq  [X\dashind ]^{-1}(\widehat{\K(X)\phi(A)\K(X)})
\end{equation}
with equality if $A$ is liminal (note that if   $A$ is liminal
 then   $\K(X)$ is also  liminal being Morita-Rieffel equivalent to the liminal 
$C^*$-algebra $\langle X,X\rangle _A\subseteq A$). Hence it suffices to show that the sets in the  right hand 
sides of \eqref{domain estimation} and \eqref{domain estimation2} coincide. However, for any 
representation  $\pi$ of $A$ and any $C^*$-subalgebra $B\subseteq \K(X)$ we have 
$$
B\subseteq \ker (X\dashind(\pi)) \,\, \Longleftrightarrow   \,\, \pi(\langle BX, B X\rangle_A)=0 \,\, \Longleftrightarrow   \,\, \langle X, B X\rangle_A  \subseteq \ker \pi. 
$$
 Thus the assertion  follows from the equality 
$$
\langle  X,\K(X)\phi(A)\K(X)X\rangle _A= \langle  \K(X)X,\phi(A)\K(X)X\rangle _A= \langle  X,\phi(A)X\rangle _A.
$$
\end{proof}

In view of Proposition \ref{regular proposition},  if  $X$ and $Y$ are regular $C^*$-correspondences 
with coefficients in $A$, then the tensoring on the right by the identity 
$1_Y$ in $Y$ yields a homomorphism $\otimes 1_Y:\K(X) \to\K(X\otimes Y) $. With help of its dual map we   
are able to analyze the relationship between the spectra of compact operators on the level of spectrum of $A$.

 \begin{prop}\label{tensoring vs duals} 
Let $X$ and $Y$ be regular $C^*$-correspondences  with coefficients in $A$. Then we have 
\begin{equation}\label{inclusion which supposed to be equality}
 [X\dashind ]\circ \Y=\widehat{\otimes 1_Y}\circ [(X\otimes Y)\dashind ].
 \end{equation}  
In other words, the diagram of multivalued maps
$$
\xymatrix{ \SA   \ar[rr]^{(X\otimes Y)\dashind}   \ar[d]^{\Y} & &  \widehat{\K(X\otimes Y)} \ar[d]^{\widehat{\otimes 1_Y}}\\
 \SA  \ar[rr]^{X\dashind}  & & \widehat{\K(X)}   
 }
$$
is commutative,  and in particular  
$$
D([X\dashind ]\circ \Y)=D(\widehat{\otimes 1_Y}\circ [(X\otimes Y)\dashind ])= 
\Y^{-1}(\widehat{\langle X, X\rangle}_A).$$  
\end{prop}
\begin{proof}
Let $\pi_A:A\to \B(H)$  be an irreducible representation. If   $\pi\in\Y_0(\pi_A)$, then   $H_\pi$ is a closed subspace of $Y\otimes_{\pi_A} H$ irreducible under  the left multiplication by elements of $A$, or more precisely, irreducible for $(Y\dashind (\pi_A))(\phi_Y(A))$. Since the tensor product of $C^*$-correspondences is both  associative and  distributive  with respect to direct sums, we may naturally identify $X\otimes_{\pi} H_{\pi}$ with a closed subspace of $X\otimes Y\otimes_{\pi_A} H$. 
Since for  $a\in \K(X)$ we have
$$
((X\otimes Y)\dashind(\pi_A))(a\otimes 1_Y)(x\otimes y\otimes_{\pi_A} h)= ax\otimes y\otimes_{\pi_A} h,
$$
 we see that the action of $((X\otimes Y)\dashind(\pi_A))(a\otimes 1_Y)$ on $X\otimes_{\pi} H_{\pi}$ coincides with the action of $(X\dashind(\pi))(a)$. In particular,  the subspace $X\otimes_{\pi} H_{\pi}$   is either $\{0\}$,  when $\pi \notin \widehat{\langle X, X\rangle}_A$, or  is irreducible for  
$((X\otimes Y)\dashind(\pi_A))(\K(X)\otimes 1_Y)$.   Consequently,  
$$
(X\dashind)\circ \Y_0(\pi_A)\subseteq\widehat{\otimes 1_Y}_0\circ (X\otimes Y)\dashind (\pi_A). 
$$
To show the reverse inclusion,  let $\rho \in  (\widehat{\otimes 1_Y})_0\circ (X\otimes Y)\dashind (\pi_A)$. 
Then  $\rho$ is an irreducible subrepresentation of  the representation $\pi_{\K(X)}:\K(X)\to \B(X\otimes Y\otimes_{\pi_A} H)$, where  $\pi_{\K(X)}(a)=((X\otimes Y)\dashind(\pi_A))(a\otimes 1_Y)$. We may consider 
the dual  $C^*$-correspondence $\widetilde{X}$ (not to be  confused with the dual $\X$ to the  
$C^*$-correspondence $X$) as an $\langle X, X\rangle_A$-$\K(X)$-imprimitivity bimodule. Then using the 
natural isomorphism  
$$
(\widetilde{X}\otimes_{\K(X)}\otimes X) \otimes Y \otimes_{\pi_A}  H \cong Y\otimes_{\pi_A}  H,
$$
cf.  \cite[Proposition 2.28]{morita}, we see that $\widetilde{X}\dashind(\pi_{\K(X)})$ is equivalent to $Y\dashind(\pi_A)\circ \phi_Y:A\to \B(Y\otimes_{\pi_A}X)$.
Since induction respects direct sums   \cite[Proposition 2.69]{morita},  
   $\widetilde{X}\dashind(\rho)$ is  equivalent to an irreducible subrepresentation $\pi$ of $Y\dashind (\pi_A)\circ \phi_Y$. Then  $\pi$ belongs to both $\widehat{\langle X, X\rangle}_A$ and  $\Y_0(\pi_A)$, and we have 
   $$
   \rho \cong X\dashind(\widetilde{X}\dashind (\rho))\cong  X\dashind (\pi).
   $$ 
Consequently, $\widehat{\otimes 1_Y}_0\circ (X\otimes Y)\dashind (\pi_A) \subseteq  
X\dashind \circ \Y_0(\pi_A)$.
\end{proof}

\begin{cor}\label{cor on composition}
 The composition of  duals to $C^*$-correspondences  coincides with the dual of their tensor product:  
$$
\X\circ \Y= \widehat{X\otimes Y}. 
$$
\end{cor}
\begin{proof}
We showed in  the proof of Proposition \ref{tensoring vs duals} that $X\dashind \circ \Y_0=\widehat{\otimes 1_Y}_0\circ (X\otimes Y)\dashind$, and all subspaces of $X\otimes Y\otimes_{\pi_A} H$ irreducible for $(X\otimes Y)\dashind(\pi_A)(\K(X)\otimes 1_Y)$ are of the form  $X\otimes_{\pi} H_{\pi}$, where $\pi\in\Y_0(\pi_A)\cap \widehat{\langle X, X\rangle}_A$. Since $\phi_{X\otimes Y}(A)\subseteq \K(X)\otimes 1_Y$,  the action of  $(X\otimes Y)\dashind(\pi_A)(\phi_{X\otimes Y}(a))$, $a\in A$,  coincides on   $X\otimes_{\pi} H_{\pi}$ with  $ X\dashind (\pi)(\phi_X(a))$. Thus we have 
$$
\X_0 \circ \Y_0=(\widehat{\phi_X}_0 \circ X\dashind) \circ \Y_0= \widehat{\phi_{X\otimes Y}}_0
\circ (X\otimes Y)\dashind= \widehat{X\otimes Y}_0. 
$$
\end{proof}


\subsection{Semigroups dual to regular product systems}

Let $X$ be  a product system over $P$. By Corollary \ref{cor on composition},    the family $\{\X_p\}_{p\in P}$  of dual maps to $C^*$-correspondences $X_p$, $p\in P$, forms a semigroup  of multivalued maps on $\SA$, that is 
$$
\X_e=id,\quad\textrm{and}\quad \X_p\circ \X_q= \X_{pq}, \qquad p,q\in P.
$$
If  $A$ is liminal then these multivalued maps 
are continuous by Proposition \ref{basic properties of the dual maps}. 

\begin{defn}
We call the \emph{semigroup} $\X:=\{\X_p\}_{p\in P}$  \emph{dual to the product system} $X$. 
\end{defn}

In the remainder of this subsection we prove certain technical facts concerning the interaction among 
Cuntz-Pimsner representations, dual maps and the process of induction. 

\begin{lem}\label{lemma instead of diagrams} Let  $X$ be  a product system over  a left 
cancellative semigroup $P$. If  $p, q, s \in P$ are such that $s \geq p, q$, then 
$$
\X_{q^{-1}s}\X_{p^{-1}s}^{-1}=[X_q\dashind^{-1}]\circ \widehat{i_{q}^{s}} \circ 
\widehat{i_{p}^{s}}^{-1} \circ [ X_p\dashind ]. 
$$
\end{lem}
\begin{proof}
Applying Proposition \ref{tensoring vs duals} to  $Y=X_{p^{-1}s}$,   $X=X_p$ and  
$Y=X_{q^{-1}s}$, $X=X_q$, respectively,   we get  
$$ 
 [X_{s}\dashind ] \X_{p^{-1}s}=\widehat{i_{p}^{s}} [X_{s}\dashind ] \quad \textrm{ and } \quad   [X_{s}\dashind ] \X_{q^{-1}s}= \widehat{i_{q}^{s}} [X_{s}\dashind ].
$$
As $[X_{s}\dashind ]$ is a homeomorphism,  this is equivalent to
 $$
 \X_{p^{-1}s}= [X_{s}\dashind ]^{-1}\widehat{i_{p}^{s}}[X_{s}\dashind ] \quad \textrm{ and } \quad 
\X_{q^{-1}s}=[X_{s}\dashind ]^{-1}\widehat{i_{q}^{s}}[X_{s}\dashind ], 
$$
and the  assertion follows.
\end{proof}


The  following Lemma  \ref{lemma for imprimitivity} is virtually  a special case of    \cite[Lemma 1.3]{kwa}. 

\begin{lem}\label{lemma for imprimitivity} 
Suppose $Y$ is an imprimitivity  Hilbert $A$-$B$-bimodule and $(\pi_A,\pi_Y, \pi_B)$ is its representation on 
a Hilbert space $H$. Thus $\pi_A:A\to \B(H)$, $\pi_B:B\to \B(H)$ are representations and with 
the map $\pi_Y:Y\to \B(H)$ they satisfy
$$
\pi_A(a)\pi_Y(y)\pi_B(b)=\pi_Y(ayb),\quad \pi_Y(x)\pi_Y(y)^*=\pi_A({_A\langle} x, y\rangle), \quad \pi_Y(x)^*\pi_Y(y)=\pi_B(\langle x, y\rangle_B),
$$
$a\in A$, $b\in B$, $x,y\in Y$. If $\pi$ is an irreducible subrepresentation of $\pi_B$ then the restriction 
$\rho(a):= \pi_A(a)|_{\pi_Y(Y) H_\pi}$ yields an irreducible subrepresentation  of $\pi_A$ such that 
$ [\rho]=[ Y\dashind (\pi)]$.
\end{lem}
\begin{proof}
Let $\pi\leq \pi_B$ be a representation of $B$ on a Hilbert space $H_\pi \subset H$. The Hilbert space $\pi_Y(Y) H_\pi\subset H$ is invariant for elements of $\pi_A(A)$ and therefore $\rho(a):= \pi_A(a)|_{\pi_Y(Y) H_\pi}$, $a\in A$ defines a representation of $A$. Since 
$$
\|\sum_{i=1}^n\pi_{Y}(y_i)h_i\|^2= \sum_{i,j=1}^n\langle \pi_{Y}(y_i)h_i,\pi_{Y}(y_j)h_j \rangle = \sum_{i,j=1}^n \langle h_i,\pi_A(\langle y_i, y_j\rangle_A) h_j \rangle= \|\sum_{i=1}^n y_i\otimes_{\pi} h_i\|^2,
$$
the mapping      $\pi_{Y}(y)h \mapsto y \otimes_\pi h$, $y\in Y$, $h\in H_\pi$, extends by linearity and continuity to a unitary  operator $V:\pi_Y(Y)H_\pi \to Y \otimes_\pi H_\pi$, which  intertwines $\rho$ and $Y\dashind (\pi)$ because 
$$
V\rho(a)\pi_{Y}(y)h =V \pi_{Y}(ay) h= (ay \otimes_\pi h)= Y\dashind (\pi)(a)  V \pi_{Y}(y)h.
$$  
Accordingly, if $\pi$ is irreducible then $\rho$, being unitary equivalent to the irreducible representation $Y\dashind (\pi)$, is also irreducible.
\end{proof}

A counterpart of \cite[Lemma 1.3]{kwa} suitable for our purposes is the following statement.

\begin{lem}\label{representation structure lemma2}
Suppose $\psi$ is a Cuntz-Pimsner covariant representation of a regular product system $X$ over $P$ on a Hilbert space $H$. Let $p,q\in P$  and let $\pi$ be an irreducible summand of $\psi^{(q)}$ acting on a subspace 
$K$ of $H$. Then the restriction 
\begin{equation}\label{in dust representation}
\pi_p(T):= \psi^{(p)}(a)|_{\psi_{p}(X_p)\psi_q(X_q)^*K}, \qquad T\in\K(X_p),
 \end{equation}
yields a  representation $\pi_p:\K(X_p)\to \B(\psi_{p}(X_p)\psi_q(X_q)^*K)$ which is either 
zero or irreducible, and such that
 $$
 [\pi_p]= [  (X_p\dashind) ((X_q\dashind)^{-1} (\pi))].
 $$
 \end{lem}
 \begin{proof} 
The dual $C^*$-correspondence $\widetilde{X}_q$  to $X_q$ is an imprimitivity $\langle X_q,X_q\rangle_A$-$\K(X_q)$-bimodule and $( \psi_e,\widetilde{\psi}_q, \psi^{(q)})$,  where $\widetilde{\psi}_q(\flat(x))=\psi_q(x)^*$, is its representation. Thus, by Lemma \ref{lemma for imprimitivity}, the restriction $\pi_e(a):=\psi_e(a)|_{\psi_q(X_q)^*K}$, $a\in A$, yields an irreducible subrepresentation $\pi_e:A\to \B(\psi_q(X_q)^*K)$  of $\psi_e$ such that $[ \pi_e]=[\widetilde{X}_q\dashind(\pi)]=[(X_q\dashind)^{-1}(\pi)]$. If $\pi_e(\langle X_p,X_p\rangle_A)=0$, then \eqref{in dust representation} is a zero representation. Otherwise  we may apply  Lemma \ref{lemma for imprimitivity} to $\pi_e$ and the representation  $( \psi^{(p)}, \psi_p, \psi_e)$   of the imprimitivity $\K(X_p)$-$\langle X_p,X_p\rangle_A$-bimodule $X_p$. Then  we see that \eqref{in dust representation} yields an irreducible representation  such that $[\pi_p]=[  X_p\dashind (\pi_e)]=[  X_p\dashind ((X_q\dashind)^{-1} (\pi))] $.
 \end{proof}


\section{A uniqueness theorem and simplicity criteria for Cuntz-Pimsner algebras}
\label{A uniqueness theorem and simplicity criteria for Cuntz-Pimsner algebras} 

Throughout this section, we consider a directed, left cancellative semigroup  $P$ and a regular product 
system $X$ over $P$ with coefficients in an arbitrary $C^*$-algebra $A$. 
We recall from Theorem \ref{structure theorem} that the Cuntz-Pimsner algebra $\OO_X$ is graded over the 
enveloping group $G(P)$ with fibers 
$$
(\OO_X)_g=\clsp\{ j_X(x)j_X(y)^*:  x,y\in X, \,\, [d(x),d(y)]=g\}, \qquad g\in G(P). 
$$
Moreover, cf. Remark \ref{remarkable remarks}, $\OO_X$   may be viewed as a full cross-sectional algebra $C^*(\{(\OO_X)_g\}_{g\in G(P)})$ of the Fell bundle $\{(\OO_X)_g\}_{g\in G(P)}$, and the 
\emph{reduced Cuntz-Pimsner algebra}  
$$
\OO_X^r:=C^*_r(\{(\OO_X)_g\}_{g\in G(P)})
$$
is defined as the reduced cross-sectional algebra  of $\{(\OO_X)_g\}_{g\in G(P)}$. 
There exists a  canonical epimorphism 
\begin{equation}\label{amenabilty epimorphism}
\lambda:\OO_X \to  \OO_X^r. 
\end{equation}
This epimorphism may not be injective. However, $\lambda$ is always injective whenever  group $G(P)$ is 
amenable or more generally when  the Fell bundle
  $
  \left\{(\OO_{X})_g\right\}_{g\in G(P)}
  $
   has the   approximation property defined in \cite{Exel}.  

We want to clarify what we mean by a {\em uniqueness theorem} in this context. By now, several conditions 
implying amenability of  the Fell  bundle $\{(\OO_X)_g\}_{g\in G(P)}$ are known. That is, conditions which guarantee the identity $\OO_X= \OO_X^r$, see e.g. \cite{KLQ}, \cite{CLSV}, \cite{Exel}. 
These conditions seem to be independent of  aperiodicity we want to investigate, and thus we decided not 
to assume any  of them. Accordingly, we seek  an intrinsic condition on the product system $X$ (or on the 
dual semigroup $\X$) which would guarantee that every Cuntz-Pimsner representation of $X$ injective on 
the coefficient algebra $A$  generates the $C^*$-algebra  lying in between $\OO_X$ and  $\OO_X^r$. 
Before proceeding further, we summarize a few know facts useful in the aforementioned context. 

\begin{prop}\label{thm:projective property}
Suppose that  $\psi$ is an injective Cuntz--Pimsner representation of a regular product system $X$. If the epimorphism $\lambda$ from \eqref{amenabilty epimorphism} is an isomorphism, then the following 
conditions are equivalent. 
\begin{itemize}
\item[i)] The canonical epimorphism
$
\Pi_\psi:\OO_X\to C^*(\psi(X))$, where  $i_X(x)=\psi(x)$,  $x\in X$, is an isomorphism.
\item[ii)] There is a coaction
$\beta$ of $G=G(P)$ on $C^*(\psi(X))$ such that
$\beta(\psi(x)) = \psi(x) \otimes i_G(d(x))$, $x \in X$.
\item[iii)] There is a conditional expectation $E_\psi$ from $C^*(\psi(X))$  onto 
$$
\FF_\psi= \clsp\{\, \psi(x) \psi(y)^* : x,y\in X, d(x) \sim d(y) \,\},
$$
 vanishing on elements $\psi(x) \psi(y)^*$ with $d(x) \nsim d(y)$, cf. Remark \ref{kernel of Ore}.
\end{itemize}
Not assuming injectivity of $\lambda$, we have implications 
(i) $\Rightarrow$ (ii) $\Rightarrow$ (iii), and (iii) is equivalent to  existence of a unique epimorphism 
$\pi_\psi: C^*(\psi(X))\to \OO_X^r$ such that the following diagram
\begin{equation}\label{uniqueness diagram}
\xymatrix{  \OO_X  \ar@/_2pc/[rr]^{\lambda\,\,} \ar[r]^{\Pi_\psi\,\,}    
&C^*(\psi(X))  \ar[r]^{\pi_\psi}  &   \OO_X^r }
\end{equation}
is commutative.
\end{prop}
\begin{proof}
It suffices to  prove the second part of the assertion.  Implication (i) $\Rightarrow$ (ii) is obvious because we know that $\OO_X$ is equipped with the coaction in the prescribed form. Suppose (ii) holds. Using the contractive projections  onto the spectral subspaces for the coaction $\beta$, cf. \cite[Lemma 1.3]{qui:discrete coactions}, 
and the fact that  elements of the form $\psi(x) \psi(y)^*$ span a dense subspace of $C^*(\psi(E))$, Lemma \ref{going forward lemma2}, we get  
$$
[C^*(\psi(X))]^\beta_g=\{c \in  C^*(\psi(X)):\beta(c)= c \otimes i_G(g)\}
=\clsp\{\psi(x) \psi(y)^*: [d(x),d(y)]=g \}.
$$
In particular,  the projection onto $[C^*(\psi(X))]^\beta_e=\FF_\psi$ is the  conditional expectation described in (iii). If we assume (iii), then $\{\Pi_\psi((\OO_X)_g)\}_{g\in G}$ is a Fell bundle which yields a  topological grading of $C^*(\psi(X))$, see \cite[Definition 3.4]{Exel}. Hence by \cite[Theorem 3.3]{Exel} there exists a desired epimorphism $\pi_\psi: C^*(\psi(X))\to \OO_X^r$. Conversely, if such an epimorphism $\pi_\psi: C^*(\psi(X))\to \OO_X^r$ exists, then composing it with the canonical conditional expectation on $\OO_X^r$ one gets the conditional expectation described in (ii).
\end{proof}

The authors  of \cite{CLSV}  call a representation
$\psi \colon X \to B$ possessing the  property described in  part (ii) of  Proposition \ref{thm:projective property}  \emph{gauge-compatible}. For our purposes the property given in part (iii) of  Proposition 
\ref{thm:projective property}  is more relevant, and thus we coin the following definition, 
cf. \cite[Definition 3.4]{Exel}.

\begin{defn}
We say that a representation $\psi \colon X \to B$ of a product system $X$ is \emph{topologically graded} 
if it has the property described in part (iii) of Proposition \ref{thm:projective property}.
\end{defn}

Thus, to conclude our discussion, by \emph{uniqueness theorem} for $\OO_X$ we understand a result which guarantees that for every injective Cuntz-Pimsner covariant representation $\psi$ of $X$ there is a map 
$\pi_\psi$ making the diagram \eqref{uniqueness diagram} commutative. By  Proposition 
\ref{thm:projective property}, this is equivalent to  $\psi$ being topologically graded.  
 We now introduce a dynamical condition which entails such a result.

\begin{defn}\label{topological freeness definition}
We say that a regular product system $X$,  or  the dual semigroup $\{\X_p\}_{p\in P}$, is \emph{topologically aperiodic}  if for each nonempty open set $U\subseteq \SA$,  each finite set  $F\subseteq  P$ and 
element $q\in P$ such that $q\nsim_R p$ for $p \in F$, there exists a $[\pi]\in U$  such that  for a 
certain enumeration of elements of $F=\{p_1,...,p_n\}$  and certain elements $s_1,...,s_n\in P$  
with  $q \leq s_1 \leq  ... \leq s_n$ and    $p_i\leq s_i$   we have    
\begin{equation}\label{topological freeness condition}
[\pi] \notin \X_{q^{-1}s_i}(\X_{p_i^{-1}s_i}^{-1}([\pi])) \qquad\textrm{for all }\,\, i=1,...,n.
\end{equation}
\end{defn}

\begin{rem}
Since  $(P, \leq)$ is a directed preorder, for any $F=\{p_1,...,p_n\}\subseteq  P$ and $q\in P$ there exists an increasing sequence $q \leq s_1 \leq  ... \leq s_n$  such that $p_i\leq s_i$ for all $i=1,...,n$. Therefore the 
essential part of the condition  in Definition \ref{topological freeness definition} is existence of a $[\pi]$ satisfying  \eqref{topological freeness condition}, which a priori  depends on the choice of the sequence  $q \leq s_1 \leq  ... \leq s_n$ and enumeration of elements of $F$.
\end{rem}

\begin{prop}\label{special cases proposition} If condition \eqref{topological freeness condition} holds for a certain sequence $q \leq s_1 \leq  ... \leq s_n$, then it also holds for any 
sequence $q \leq s_1' \leq  ... \leq s_n'$ such that 
$$
 p_i \leq s_i' \leq s_i\quad 
\textrm{ for all } i=1,...,n.
$$
Moreover, we have the following. 
\begin{itemize}
\item[i)] If $(G(P),P)$ is a quasi-lattice ordered group then in Definition \ref{topological freeness definition} 
one can always take 
$$
s_1=p_1\vee q \quad\textrm{ and  } \quad s_i=p_i \vee s_{i-1}\quad  \textrm{ for all }i=2,...,n.
$$
\item[ii)] Topological aperiodicity of $X$ implies that  for any open nonempty  set $U\subseteq \SA$ and 
any finite set  $F\subseteq  P$ such that $p\nsim_R e$ for $p\in F$, there is a 
$[\pi]\in U $ satisfying
\begin{equation}\label{positive aperiodicity}
[\pi] \notin \X_{p}([\pi])  \quad \textrm{ for all } p \in F.
\end{equation}
  If $(P, \leq )$ is  linearly ordered then the  converse implication also holds.
\item[iii)] In the simplest case of a product system $\{X^{\otimes n}\}_{n\in \N}$ arising from a single regular $C^*$-correspondence $X$, the topological aperiodicity  is equivalent to that
 for each $n>0$ set
 $$
 F_n=\{ [\pi] \in \SA:   \pi \in \X^{n}([\pi])\}
 $$
 has empty interior. (In this case we will say that the $C^*$-correspondence $X$ is topologically aperiodic.)
\end{itemize}
\end{prop}
\begin{proof}
 Let us notice  that  if $q, p_i \leq s_i' \leq s_i$, then using the semigroup property of $\X$  (Corollary \ref{cor on composition}), surjectivity of mappings $\X_p$, $p \in P$,  (Proposition \ref{basic properties of the dual maps}) and taking into account  \eqref{multivalued identity} we get
\begin{align*}
\X_{q^{-1}s_i}\circ \X_{p_i^{-1}s_i}^{-1}
&=\X_{q^{-1}s_i'} \circ \X_{s_i'^{-1}s_i} \circ (\X_{p_i^{-1}s_i'} \circ\X_{s_i'^{-1}s_i}  )^{-1}
\\
&=\X_{q^{-1}s_i'} \circ \X_{s_i'^{-1}s_i} \circ  \X_{s_i'^{-1}s_i}^{-1} \circ \X_{p_i^{-1}s_i'}^{-1}
\\
&\supseteq  \X_{q^{-1}s_i'}  \circ \X_{p_i^{-1}s_i'}^{-1}.
\end{align*}
Hence $[\pi] \notin \X_{q^{-1}s_i}(\X_{p_i^{-1}s_i}^{-1}([\pi]))$ implies $[\pi] \notin \X_{q^{-1}s_i'} ( \X_{p^{-1}s_i'}^{-1}([\pi]))$. This  proves the initial part of the assertion. 
\\
Ad (i). It follows immediately from what we have just shown.
\\
Ad (ii). If   $F=\{p_1,...,p_n\}\subseteq  P$  and $p\nsim_R e$ for  all $p\in F$, then putting $q=e$ we see that  topological aperiodicity of $X$ implies that for any nonempty open   set $U\subseteq \SA$  there are elements  $s_1,...,s_n\in P$, $p_i\leq s_i$, $i=1,...,n$  and a point $[\pi]\in U$ such that 
$$
[\pi] \notin \X_{q^{-1}s_i}(\X_{p_i^{-1}s_i}^{-1}([\pi]))=\X_{s_i}(\X_{p_i^{-1}s_i}^{-1}([\pi]))  \qquad\textrm{for all }\,\, i=1,...,n.
$$
By the inclusion noticed above we have  
$
\X_{s_i}\circ\X_{p_i^{-1}s_i}^{-1}=\X_{p_i}\circ\X_{p^{-1}s_i}\circ\X_{p_i^{-1}s_i}^{-1}\supseteq \X_{p_i}
$
and thus condition \eqref{positive aperiodicity} follows.

Conversely, suppose $(P, \leq )$ is  linearly ordered and the condition described in (ii) is satisfied. Let  $U\subseteq \SA$ be open and nonempty, $F\subseteq  P$ finite and  $q\in P$ such that  $q\nsim_R p$, for $p \in F$. Enumerating elements of $F=\{p_1,...,p_n\}\subseteq  P$ in a non-increasing order we have  
$$ 
p_1\leq p_2 \leq ... \leq p_{k_0}\leq q \leq p_{k_0+1}\leq ... \leq p_n
$$ 
for certain $k_0\in\{0,1,...,n\}$.  Defining
$$
s_i:=\begin{cases}
q ,& i\leq k_0 \\
p_i & i \geq k_0+1
\end{cases}
$$
we see that  $q \leq s_1 \leq  ... \leq s_n$ and
$$
\X_{q^{-1}s_i}\circ \X_{p_i^{-1}s_i}^{-1}=\begin{cases}
\X_{p_i^{-1}q}^{-1} ,& i\leq k_0 \\
\X_{q^{-1}p_i} & i \geq k_0+1
\end{cases}.
$$
Put $F':=\{ p_i^{-1}q: i=1,...k_0\}\cup \{ q^{-1}p_i: i=k_0+1,...n\}$ and note that   $p\nsim_R e$ 
for all $p\in F'$. Thus we may apply condition described 
in (ii) to $F'$ and then we obtain a $[\pi]\in U$ satisfying \eqref{topological freeness condition}.

Ad (iii). By part (ii) above, topological aperiodicity implies the condition described in (iii). To see the converse, again by part (ii), it suffices to show \eqref{positive aperiodicity} for a finite set $F\subseteq \N\setminus\{0\}$. The latter follows from condition described in (iii) applied to $n=m!$ where $m=\max\{k: k\in  F\}$.
\end{proof}

Now, we are ready to state and prove the main result of the present paper. 
  

 \begin{thm}[Uniqueness theorem]\label{Cuntz-Krieger uniqueness theorem} 
Suppose that a  regular product system $X$ is topologically aperiodic. Then every injective Cuntz-Pimsner 
representation of $X$ is topologically graded.  If the canonical epimorphism $\lambda:\OO_X \to  \OO_X^r$ 
is injective then there is a natural isomorphism
  $$
  \OO_{X}\cong  C^*(\psi(X))
  $$
 for every injective Cuntz-Pimsner representation $\psi$ of $X$. 
\end{thm}
\begin{proof}
Suppose that $\psi$ is an injective Cuntz-Pimsner representation of $X$ in a $C^*$-algebra $B$. Then 
$\psi^{(p)}:\K(X_p)\to B$ is injective for all $p\in P$. Let us consider an element of the form   
\begin{equation}\label{general form33} 
\psi^{(q)}(S_{q}) + \sum_{p\in F} \psi_{p,q}(S_{p,q}),
\end{equation}
where $q\in P$, $F\subseteq P$ is a finite set such that $q \nsim_R p$ for all $p\in F$, and $S_q\in \K(X_q)$,  $S_{p,q}\in \K(X_q, X_p)$. By Lemma \ref{going forward lemma2} part (iii), such elements form a 
dense subspace of $C^*(\psi(X))$. Thus  existence of the appropriate conditional expectation 
will follow from the inequality 
$$
\|S_q\|=\|\psi^{(q)}(S_{q})\| \leq   \|\psi^{(q)}(S_{q}) + \sum_{p\in F} \psi_{p,q}(S_{p,q})\|.
$$
To prove this inequality,  we fix $\varepsilon >0$ and recall  that for any $a\in A$ the mapping 
$\SA\ni [\pi] \mapsto \|\pi(a)\|$ is lower semicontinuous and attains its maximum equal to $\|a\|$, cf. e.g. \cite[Proposition 3.3.2., Lemma 3.3.6]{Dix:C*-algebras}. Thus, since $X_{q}\dashind:\SA \to \widehat{\K(X)}$ is a homeomorphism, we deduce that   there is an open nonempty set $U\subseteq \SA$ such that 
$$
\|X_{q}\dashind(\pi)(S_{q})\| >  \Vert S_{q} \Vert - \varepsilon \qquad \textrm{for  every } [\pi]\in U.
$$
Let $F=\{p_1,...,p_n\}$. By topological aperiodicity  of $X$,  there are    elements $s_1,...,s_n\in P$ such that 
$q \leq s_1 \leq  ... \leq s_n$ and   $p_i\leq s_i$, $i=1,...,n$,  and  there exists a $[\pi]\in U$ satisfying \eqref{topological freeness condition}.  
Let us fix these objects.

 We recall that if $p\leq s$, then $i_{p}^{s}(\K(X_{p})) \subseteq \K(X_{s})$ and  thus $\psi^{(p)}(\K(X_{s}))\subseteq \psi^{(s)}(\K(X_{s}))$, cf. Lemma \ref{going forward lemma}. In particular, we have the increasing sequence of algebras
 $$
  \psi^{(q)}(\K(X_{q}))\subseteq \psi^{(s_1)}(\K(X_{s_1}))\subseteq ... 
\subseteq \psi^{(s_n)}(\K(X_{s_n})) \subseteq C^*(\psi(X)).
 $$
We construct a relevant  sequence of  representations of these algebras as follows. We put 
 $$
 \nu_q:\psi^{(q)}(\K(X_{q}))\to \B( H_q) \quad \textrm{  defined as  }\quad \nu_q(\psi^{(q)}(S))=X_{q}\dashind(\pi)(S).
 $$
Then $\nu_q$ is  an irreducible representation because so is $\pi$. 
We let $\nu_{s_1}:\psi^{(s_1)}(\K(X_{s_1}))\to \B( H_{s_1})$ to be any irreducible extension of $\nu_q$, and 
for $i=2,3,...,n$ we take $\nu_{s_i}:\psi^{(s_i)}(\K(X_{s_i}))\to \B( H_{s_i})$ to be  any  irreducible 
extension of $\nu_{s_{i-1}}$. Finally, we let   $\nu:C^*(\psi(X))\to \B(H)$ to be any extension of  
$\nu_{s_n}$. In particular, we have
$$
H_q\subseteq H_{s_1}\subseteq  ... \subseteq H_{s_n}\subseteq H.
$$ 
 Let   $P_q\in \B(H)$ be  the projection onto the subspace $H_q$. Clearly 
$$ 
 \|P_q \nu(\psi^{(q)}(S_{q}))P_q\|=\|\nu_q(\psi^{(q)}(S_{q}))\|=\|X_{q}\dashind(\pi)(S_{q})\| \geq \Vert S_{q} \Vert - \varepsilon 
$$
and as $\varepsilon$ is arbitrary we can reduce the proof  to showing that   
\begin{equation}\label{equation to be proved}
P_q \nu(\psi_{p,q}(S_{p,q}))   P_q =0\quad \textrm{ for } p \in F.
\end{equation}
To this end,  we fix a $p_i\in F$. Let $P_{s_i}$ be the projection onto $H_{s_i}$ and  consider  the space 
$$
H_{p_i}:=\nu(\psi_{p_i}(X_{p_i})\psi_q(X_q)^*)H_q.
$$
 We claim that $P_{s_i}H_{p_i}=\{0\}$. Since $H_q\subseteq H_{s_i}$, this implies \eqref{equation to be proved} and finishes the proof. Suppose to the contrary that $P_{s_i}H_{p_i}\neq \{0\}$. By Lemma 
\ref{representation structure lemma2} and  the definitions of $\nu$ and $H_{p_i}$, the mapping  
$$
\K(X_{p_i})\ni S \longrightarrow \nu(\psi^{(p_i)}(S))|_{H_{p_i}}
$$  
is an irreducible representation equivalent to $X_{p_i}\dashind (\pi)$. In particular,  $H_{p_i}$ is  irreducible   
for $\nu(\psi^{(p_i)}(\K(X_{p_i})))$. Since  
$$
\nu(\psi^{(p_i)}(\K(X_{p_i})))\subseteq \nu(\psi^{(s_i)}(\K(X_{s_i})))\,\,\,\textrm{ and }\,\,\, P_{s_i} \in \nu(\psi^{(s_i)}(\K(X_{s_i})))',
$$ 
we see that $P_{s_i}H_{p_i}$ is an irreducible subspace for $\nu(\psi^{(p_i)}(\K(X_{p_i})))$. Thus, since  $H_{p_i}$ and $P_{s_i}H_{p_i}$ are both irreducible subspaces for  $\nu(\psi^{(p_i)}(\K(X_{p_i})))$,  either $H_{p_i}=P_{s_i}H_{p_i}$ or $H_{p_i} \bot P_{s_i}H_{p_i}$. However, (as $P_{s_i}H_{p_i}\neq \{0\}$) the latter is clearly impossible.
Thus  $H_{p_i} \subseteq  H_{s_i}$ and denoting by $\pi_{s_i}$ the  representation 
 $$
\K(X_{s_i})\ni S \stackrel{\pi_{s_i}}{\longrightarrow} \nu(\psi^{(s_i)}(S))|_{H_{s_i}},
$$
 we get $[\pi_{s_i}] \in  \widehat{\iota_{p_i}^{s_i}}^{-1}([ X_{p_i}\dashind (\pi)])$. 
Denoting by $\pi_q$ the  representation 
$$
\K(X_{q})\ni S \to \nu(\psi^{(q)}(S))|_{H_{q}},
$$ 
we have 
$[\pi_{q}]\in \widehat{\iota_{q}^{s_i}}([\pi_{s_i}])$  and $\pi_q=X_{q}\dashind(\pi)$.
Hence we get
$$
[\pi]=[(X_q\dashind)^{-1}(\pi_{q})]\in [X_q\dashind^{-1}](\widehat{i_{q}^{s_i}}([\pi_{s_i}]))
\subseteq  [X_q\dashind^{-1}](\widehat{i_{q}^{s_i}}(\widehat{i_{p_i}^{s_i}}^{-1}([ X_{p_i}\dashind (\pi)]))).
$$
Thereby  in view of Lemma \ref{lemma instead of diagrams} we arrive at
$$
[\pi]\in  \X_{q^{-1}s_i}(\X_{p_i^{-1}s_i}^{-1})([\pi]), 
$$
which contradicts the choice of $\pi$.
\end{proof}

As an  application of Theorem \ref{Cuntz-Krieger uniqueness theorem},  we obtain  simplicity criteria for the reduced Cuntz-Pimsner algebra $\OO_X^r$. To this end, we first introduce the indispensable terminology.

\begin{defn}\label{minimalsystem}
Let $X$ be a regular product system over a semigroup $P$ with coefficients in a $C^*$-algebra $A$. 
We say that an ideal $J$  in $A$ is \emph{$X$-invariant} if and only if for each $p\in P$ the set
$$ 
X^{-1}_p(J):=\{a\in A: \langle X_p, a X_p \rangle_p \subseteq J\}   
$$
is equal to $J$. We say  $X$ is \emph{minimal} if there are no nontrivial $X$-invariant ideals in $A$, that is if for any  ideal $J$ in $A$ we have 
$$
(\forall{p\in P}) \,\, X^{-1}_p(J)= J \,\,\,\implies J=\{0\} \,\, \textrm{ or } J=A.
$$
\end{defn}

\begin{rem}\label{remark on minimality}
When $P=\N$ and $A$ is unital, our Definition \ref{minimalsystem} agrees with 
\cite[Definition 3.7]{Szwajcar} treating the case of a single $C^*$-correspondence, see the discussion 
on page 418 therein. 
\end{rem}

\begin{rem}\label{remark on saturation}
Let $X$ be a regular $C^*$-correspondence. It is well known, cf. for instance  \cite[Proposition 1.3]{ka3}, 
that for any ideal $J$ in $A$  we have
$$
XJ= \{x j:x\in X,\, j\in J\}=\{x\in X: \langle x,y\rangle\in J \textrm{ for all }y\in X\}.
$$
Therefore we see  that 
$$
X^{-1}(J)=\{a\in A:  a X \subseteq  X J\}=\phi^{-1}(\K_J(X)),
$$
where 
$$ \K_J(X):=\clsp\{\Theta_{x,y}: x\in X,y\in XJ\}=\clsp\{\Theta_{x,y}: x,y\in XJ\}
$$ 
is an ideal in $\K(X)$. In particular, we infer that $X^{-1}(J)$ is an ideal and $J$ is $\{X^{\otimes n}\}_{n\in \N}$-invariant if and only if $X^{-1}(J)=J$, in which case we will say that $J$ is $X$-invariant.  
\end{rem}

 \begin{thm}[Simplicity of $\OO_X^r$]\label{simplicity} 
If a regular product system $X$ is topologically aperiodic and minimal, then $\OO_X^r$ is simple.
\end{thm}
\begin{proof}
Suppose $I$ is an ideal in $\OO_X^r$ and put $J=(j_X^r)^{-1}(I)\cap A=\{a\in A: j_X^{r}(a)\in I\}$. 
Then $J$ is an ideal in $A$. We claim that $J$  is  $X$-invariant. Indeed, for $p\in P$ we have 
$$
j_X^{r}(\langle X_p, J X_p \rangle_p)=j_X^{r}( X_p)^*  j_X^{r}(J X_p)=j_X^{r}( X_p)^*  j_X^{r}(J)  j_X^{r} (X_p)\subseteq I. 
$$
That is, $\langle X_p, J X_p \rangle_A \subseteq J$ and hence $J\subseteq X^{-1}_p(J)$. 
On the other hand, if $a\in    X^{-1}_p(J)$ then by Remark \ref{remark on saturation} we have  
$$
\phi_p(a)=\sum_{i} \Theta_{x_i, y_i j_i}\,\,\,\textrm{  where } \,\,\,x_i,y_i\in X_p \,\,\textrm{ and } j_i\in J.
$$
Since $j_X^r:X\to \OO_X^r$ is Cuntz-Pimsner covariant, we get
\begin{align*}
j_X^r(a)&={j_X^r}^{(p)}(\phi_p(a))=\sum_{i} {j_X^r}^{(p)}(\Theta_{x_i, y_i j_i})=
\sum_{i} j_X^r(x_i) j_X^r( y_i j_i)^*
\\
&=\sum_{i} j_X^r(x_i) j_X^r(j_i^*)j_X^r(y_i)^*\in I.
\end{align*}
 Thus $X^{-1}_p(J) \subseteq J$ and this proves our claim. In view of minimality of $X$,  either $J=A$ 
or $J=\{0\}$. In the former case, $\OO_X^r=C^*(j_X^r(X))=I$ because $j_X^r(X_p)=j_X^r(AX_p)=j_x^r(A)j_X^r(X_p)\subseteq I$ for each $p\in P$. 
In the latter case, the composition of $j_X^r:X \to  \OO_X^r$ with the quotient map 
$\theta:\OO_X^r \to \OO_X^r/I$ yields a  Cuntz-Pimsner  representation $k_X:=\theta\circ j_X^r$ of $X$ in $\OO_X^r/I$ which is injective on $A$. Thus by Theorem \ref{Cuntz-Krieger uniqueness theorem} 
we have an epimorphism 
$$
\pi_{k_X}:\OO_X^r/I \to \OO_X^r
$$
such that $\pi_{k_X}(q(j_X^r(x))=j_X^r(x)$, $x\in X$. Hence $j_X^r(X)\cap I=\{0\}$ and therefore $I=\{0\}$.
\end{proof}
 
Schweizer  found in \cite{Szwajcar} a necessary and sufficient condition for simplicity of Cuntz-Pimsner algebras associated with single $C^*$-correspondences,   improving similar results of \cite{KPW}. Namely, by \cite[Theorem 3.9]{Szwajcar},  if $X$ is a left essential and full $C^*$-correspondence with coefficients in a unital $C^*$-algebra $A$,  then   $\OO_{X}$ is simple if and only if 
 $X$ is minimal and \emph{nonperiodic}, meaning that $X^{\otimes n}\approx {_A} A_A$ implies $n=0$, where  
$\approx$ denotes  the unitary equivalence of $C^*$-correspondences.
This result  suggests that  topological aperiodicity of a product system $X$ should imply  nonperiodicity of $X$, and this  is indeed the case. 

\begin{prop}
Suppose that $X$ is a topologically aperiodic regular product system over a semigroup $P$ of Ore type. Then 
$X_{p} \approx  X_e$ implies $p\sim_R e$, and if in addition $(G(P),P)$ is a quasi-lattice ordered group, then $
\,\,X_{q^{-1}(p\vee q)} \approx X_{p^{-1}(p\vee q)}$  implies  $p=q$.
\end{prop}
\begin{proof}
In view of Proposition \ref{special cases proposition} parts (i) and (ii), it suffices to note that  $X_p \approx X_q$  implies that  $[\pi]\in \X_{p}(\X_{q}^{-1}([\pi]))$ for all $[\pi]\in \SA$. To this end, let  $V:X_p\to X_q$ be a bimodule unitary implementing the equivalence $X_p \approx X_q$.  Let $[\pi]\in \SA$  be arbitrary and take any  $[\rho]\in \X_{q}^{-1}([\pi])$ (such $\rho$ exists because $\X_q$ is surjective). In other words,  
$[\pi] \preceq [X_q\dashind (\rho)\circ \phi_q]$. Then $V$ gives rise to a unitary map
$$
\widetilde{V}:X_p\otimes_\rho H_\rho \to X_q\otimes_\rho H_\rho, \qquad  \textrm{ such that }\;\;
\widetilde{V}(x\otimes h) = (Vx)\otimes h. 
$$
Indeed, this  follows from the following simple computation: 
\begin{align*} 
\|\sum_{i=1}^n x_i\otimes h_i\|^2
&=\sum_{i,j=1}^n \langle x_i\otimes_\rho h_i, x_j\otimes_\rho h_j\rangle=\sum_{i,j=1}^n \langle  h_i, \rho(\langle x_i, x_j\rangle_A) h_j\rangle
\\
&=\sum_{i,j=1}^n \langle  h_i, \rho(\langle V x_i, V x_j\rangle_A) h_j\rangle=\sum_{i,j=1}^n \langle (Vx_i)\otimes_\rho h_i, (Vx_j)\otimes_\rho h_j\rangle
\\
&=\|\sum_{i=1}^n (Vx_i)\otimes_\rho h_i\|^2, 
\end{align*}
where $x_i \in X_p$, $h_i\in H_\rho$, $i=1,...,n$. 
Since $V$ is a left $A$-module morphism, we see that $\V$ establishes a  unitary equivalence between  $X_p\dashind(\rho)\circ \phi_p$ and $X_q\dashind(\rho)\circ \phi_q$. Hence  we have both $[\pi] \preceq [X_q\dashind (\rho)\circ \phi_q]$ and $[\pi] \preceq [X_p\dashind (\rho)\circ \phi_p]$.   
\end{proof}


\section{Applications and examples}\label{Applications and examples}

In this section, we give several examples and applications of the theory developed above. In particular, 
we discuss algebras associated with saturated Fell bundles, twisted $C^*$-dynamical systems, product systems of topological graphs and the Cuntz algebra $\Qq_\N$. 


\subsection{Product systems of Hilbert bimodules, Fell bundles and dual partial actions}
\label{Product systems of Hilbert bimodules, Fell bundles and dual partial group actions}

In this subsection, we consider a regular product system $X$ over a semigroup $P$ of Ore type, with the additional property that each $C^*$-correspondence $X_p$, $p\in P$, is  a Hilbert bimodule equipped with left $A$-valued inner product ${_p\langle \cdot} ,\cdot \rangle: X_p\times X_p \to A $. We call such an $X$ \emph{regular product system of Hilbert bimodules}. With help of for instance \cite[Proposition 1.11]{kwa-doplicher}, one can show that a regular product system is a product system of Hilbert bimodules if and only if each left action homomorphism $\phi_p: A \to \K(X_p)$ is surjective. In this case,  $\phi_p:A\to \K(X_p)$ is an isomorphism and 
$$
{_p\langle}  x, y \rangle=\phi_p^{-1}(\Theta_{x,y}), \qquad x,y \in X_p.
$$  
The following Proposition \ref{characterization of Hilbert bimodules} gives another characterization of regular 
product systems of Hilbert bimodules in terms of  the Fell bundle structure in $\OO_X$ identified in Theorem \ref{structure theorem} above, cf. \cite[Theorem 5.9]{kwa-doplicher}.

\begin{prop}\label{characterization of Hilbert bimodules}
A regular product system $X$ over a semigroup $P$ of Ore type is a product system of Hilbert bimodules if and only if  the algebra of coefficients $A$ embeds into $\OO_X$ as the core subalgebra $(\OO_X)_{[e,e]}$, that is 
$$
j_X(A)=(\OO_X)_{[e,e]}.
$$
In this case, each space $X_p$ embeds  into $\OO_X$ as the  fiber $(\OO_X)_{[p,e]}$. In particular,  $j_X(X_p)=(\OO_X)_{[p,e]}$, for all $p\in P$, and 
\begin{equation}\label{simple form of fibres}
(\OO_X)_{[p,q]}=\clsp\{j_X(x)j_X(y)^*: x\in X_p, y\in X_q\}, \qquad p,q\in P.
\end{equation}
\end{prop}
\begin{proof} If all the maps $\phi_p:A \to \K(X_p)$ are isomorphisms, it follows from Lemma \ref{lemma on tensoring regular correspondences} part (ii) that all the maps 
$\iota^{pr,qr}_{p,q}:\K(X_q,X_p)\to \K(X_{qr},X_{pr})$ are (Banach space) isomorphisms. Hence  
$$
\underrightarrow{\,\lim\,\,} \K(X_{qr},X_{pr})=\varphi_{p,q}(\K(X_q,X_p))
$$
where $\varphi_{p,q}$  denotes the natural embedding of $\K(X_q,X_p)$ into the inductive limit $\underrightarrow{\,\lim\,\,} \K(X_{qr},X_{pr})$.  As  the isomorphism from Theorem \ref{structure theorem}
sends $j_X(x)j_X(y)^*$ to $\varphi_{p,q}(\Theta_{x,y})$, $x\in X_p, y\in X_q$, we get \eqref{simple form of fibres}. In particular, we have $j_X(A)=(\OO_X)_{[e,e]}$.
\\
Conversely, if we assume that $\phi_p:A \to \K(X_p)$ is not onto for certain $p\in P$. Then 
$$
\varphi_{e,e}(\K(A))=\varphi_{p,p}(\phi_p(A))\subsetneq \varphi_{p,p}(\K(X_p))\subseteq 
\underrightarrow{\,\lim\,\,} \K(X_{r},X_{r}),
$$ 
and hence $j_X(A)\subsetneq (\OO_X)_{[e,e]}$.
\end{proof}

\begin{rem}\label{saturated bundels} 
 If $\{B_g\}_{g\in G}$ is a saturated Fell bundle, \cite{Fell}, i.e.
$$
B_gB_{g^{-1}}=B_e, \qquad \textrm{ for all } g\in G,
$$
 we may treat $X=\bigsqcup_{g\in G} B_g$ as a regular product system  of Hilbert bimodules with the 
structure inherited in an obvious way from $\{B_g\}_{g\in G}$. Then the Fell bundles  
$\{(\OO_X)_g\}_{g\in G}$ and $\{B_g\}_{g\in G}$ coincide. Accordingly, every cross sectional algebra of a saturated Fell bundle admits a natural realization as Cuntz-Pimsner algebra of a regular product system of Hilbert bimodules. Conversely, by Proposition \ref{characterization of Hilbert bimodules}, if $X$ is a regular product system of Hilbert bimodules over a semigroup $P$ of Ore type,  and  each fiber   $X_p$  is  nondegenerate as 
right Hilbert module  (so it is an imprimitivity bimodule), then  the  Fell bundle 
$\{(\OO_X)_{g}\}_{g\in G(P)}$ is saturated.
\end{rem}

Suppose $X$ is a regular product system of Hilbert bimodules over a semigroup $P$ of Ore type. Since all the maps $\phi_p: A \to \K(X_p)$, $p\in P$, are isomorphisms, we infer  from Definition \ref{dual map def}
that the semigroup $\X=\{\X_p\}_{p\in P}$ dual  to $X$  consists of  partial homeomorphisms $\X_p$ with domain $\widehat{\langle X_p, X_p\rangle}_p$ and range $\SA$. We show in Proposition 
\ref{partial dynamical system out of the product system} below that the semigroup $\{\X_p\}_{p\in P}$  
generates a partial action of the enveloping group $G(P)$. 
We recall the  relevant definitions concerning partial actions, cf. e.g. \cite{ELQ}.

\begin{defn}
A  \emph{partial action} of a group $G$ on a topological space $\Omega$ consists of a pair $(\{D_g\}_{g\in G}, \{\theta_g\}_{g\in G}) $, where $D_g$'s  are open subets of $\Omega$ 
and   $\theta_g:D_{g^{-1}}\to D_{g}$ are homeomorphisms such that 
\begin{enumerate}\renewcommand{\theenumi}{PA\arabic{enumi}}
\item $D_e=\Omega$ and $\theta_e=id$,
\item $\theta_t(D_{t^{-1}}\cap D_{s})=D_{t}\cap D_{ts}$,
\item $\theta_s(\theta_t(x))=\theta_{st}(x)$, for  $x\in D_{t^{-1}}\cap D_{t^{-1}s^{-1}}$. 
\end{enumerate} 
The partial action $(\{D_g\}_{g\in G}, \{\theta_g\}_{g\in G})$ is \emph{topologically free} if for every open nonempty  $U\subseteq \Omega$ and finite $F\subseteq G\setminus\{e\}$ there exists $x\in U$ such that 
$x \in  D_{t^{-1}}$ implies $\theta_{t}(x)\neq x$ for all $t\in F$.
\end{defn}

\begin{prop}\label{partial dynamical system out of the product system}
Suppose $X$ is a regular product system of Hilbert bimodules  and the underlying semigroup $P$ is of Ore type.  
The formulas 
$$
D_{[q,p]}:=\X_q(\widehat{\langle X_p, X_p\rangle_p}), 
$$
$$
\X_{[p,q]}([\pi]):=\X_p \X_q^{-1}([\pi]),\qquad [\pi]\in D_{[q,p]}, \,\, p,q\in P,
$$
yield a well defined family of open sets $\{D_g\}_{g\in G(P)}$ and homeomorphisms 
$\X_g:D_{g^{-1}}\to D_{g}$ such that $(\{D_g\}_{g\in G(P)}, \{\X_g\}_{g\in G(P)})$ is a 
partial action of $G(P)$ on $\SA$. Moreover,
\begin{itemize}
\item[i)] $\{X_g\}_{g\in G(P)}$ is a semigroup dual to $\{(\OO_X)_{g}\}_{g\in G(P)}$, where we treat  $\{(\OO_X)_{g}\}_{g\in G(P)}$ as a product system, and  $\X_g$ are viewed as multivalued 
maps on $\SA$ with $\X_g(\SA \setminus D_{g^{-1}})=\{\emptyset\}$. 
\item[ii)] We have the following implication:
\begin{equation}\label{freeness imply aperiodicity}
 (\{D_g\}_{g\in G(P)}, \{\X_g\}_{g\in G(P)}) \textrm{ is topologically free}\,\, \Longrightarrow\,\, X \textrm{ is topologically aperiodic}, 
\end{equation}
and if  $P$ is both left and right Ore (so for instance it is a group or a cancellative abelian semigroup)  
then the above implication is actually an equivalence.
\end{itemize}
\end{prop}
\begin{proof} 
To begin with, let us note that for an  ideal $I$ in $A$ and $p\in P$ we have 
\begin{equation}\label{widehatequation}
\X_p(\widehat{I})=\widehat{{_p\langle} X_p I, X_p\rangle }, \qquad  \X_p^{-1}(\widehat{I})=
\widehat{\langle X_p , IX_p\rangle_p },  
\end{equation}
cf. \cite[Remark 2.3]{kwa}, \cite[Subsection 3.3]{morita}. Now, 
let $[\pi]\in \SA$ and $r\in P$ be arbitrary. Natural representatives of the classes $\X_p \X_q^{-1}([\pi])$ and $\X_{pr} \X_{qr}^{-1}([\pi])$  act  by multiplication  from  the left on the spaces 
$$
X_p\otimes \widetilde{X}_q\otimes_\pi H_\pi, \qquad  X_{pr}\otimes \widetilde{X}_{qr}\otimes_\pi H_\pi,
$$
respectively. The obvious $C^*$-correspondence isomorphisms 
$$X_{pr}\otimes \widetilde{X}_{qr}\cong X_{p}\otimes (X_r\otimes \widetilde{X}_r)\otimes \widetilde{X}_{q}\cong X_{p}\otimes A\otimes \widetilde{X}_{q}\cong X_{p}\otimes \widetilde{X}_{q}$$
yield a unitary equivalence between the aforementioned representations. Hence $\X_p \X_q^{-1}([\pi])=\X_{pr} \X_{qr}^{-1}([\pi])$, and thus $\X_p \X_q^{-1}$ does not depend on the choice of  representatives of $[p,q]$. 
It follows from (\ref{widehatequation}) that the natural domain of $\X_p \X_q^{-1}$ is  
$\X_q(\widehat{\langle X_p, X_p\rangle_p})$ which coincides with the spectrum of ${_q\langle} X_q\langle X_p, X_p\rangle_p,  X_q \rangle $. 
This shows that the formulas above indeed define homeomorphisms $\X_g:D_{g^{-1}}\to D_{g}$, $g\in G(P)$.

Condition (PA1) is  obvious. To show (PA2), let $t=[t_1,t_2]$,  $s=[s_1,s_2]$ and $r\geq t_2,s_1$.  Putting $q=t_1 (t_2^{-1}r)$,  $p=s_2(s_1^{-1}r)$,  we have  $t=[t_1(t_2^{-1}r),t_2(t_2^{-1}r)]=[q,r]$ and $s=[s_1(s_1^{-1}r),s_2(s_1^{-1}r)]=[r,p]$.  Hence
$$
\X_{t}(D_s)= \X_{[q,r]}(D_{[r,p]})=\X_q \X_r^{-1} (\X_r(D_{[e,p]}))=\X_q (D_{[e,p]}\cap D_{[e,r]}).
$$
On the other hand, since $st=[t_1,t_2]\circ [s_1,s_2]=[t_1(t_2^{-1}r), s_2(s_1^{-1}r)]=[q,p]$, we have
$$
D_{ts} \cap D_t =D_{[q,p]}\cap D_{[q,r]}=\X_q ( D_{[e,p]})\cap \X_q ( D_{[e,r]})= \X_q ( D_{[e,p]} \cap D_{[e,r]}).
$$
This proves condition (PA2).

To show  (PA3), let   $t=[t_1,t_2]$,  $s=[s_1,s_2]$,  $r\geq t_1,s_2$ and  $[\pi]\in D_{t^{-1}}\cap D_{t^{-1}s^{-1}}$.
Then a natural representative of $\X_{st}([\pi])=\X_{[s_1,s_2]\circ [t_1,t_2]}([\pi])=\X_{s_1 s_2^{-1}r}\X_{t_2t_1^{-1}r}^{-1}([\pi])$ acts by left multiplication on the space
$$
X_{s_1 s_2^{-1}r}\otimes \widetilde{X}_{t_2t_1^{-1}r}\otimes_\pi H_\pi=X_{s_1} \otimes X_{s_2^{-1}r}\otimes \widetilde{X}_{t_1^{-1}r}\otimes \widetilde{X}_{t_2}\otimes_\pi H_\pi, 
$$
Similarly, a representative of  $\X_{s}(\X_{t}([\pi]))=(\X_{s_1} \circ \X_{s_2}^{-1}\circ \X_{t_2}\circ \X_{t_2}^{-1})([\pi])$ acts by left multiplication on the space
$$
X_{s_1} \otimes \widetilde{X}_{s_2} \otimes X_{t_1}\otimes \widetilde{X}_{t_2}\otimes_\pi H_\pi. 
$$
The latter can be considered an invariant subspace of the former with help of the following 
natural isomorphisms of $C^*$-correspondences:
\begin{align*}
X_{s_1} \otimes \widetilde{X}_{s_2} \otimes X_{t_1}\otimes \widetilde{X}_{t_2} &\cong X_{s_1} \otimes \widetilde{X}_{s_2} \otimes (  X_{r} \otimes \widetilde{X}_{r}) \otimes X_{t_1}\otimes \widetilde{X}_{t_2}
\\
& \cong X_{s_1} \langle X_{s_2}, X_{s_2}\rangle_{s_2} \otimes  X_{s_2^{-1}r} \otimes \widetilde{X}_{t_1^{-1}r}\otimes  \langle X_{t_1},X_{t_1}\rangle_{t_1} \widetilde{X}_{t_2}.
\end{align*}
By the choice of $[\pi]$ and property (PA2), we see that $\X_{s}\X_{t}([\pi])$ is nonzero and thus equals 
$\X_{st}([\pi])$, as irreducible representations have  no non-trivial subrepresentations.

Ad (i). This follows  from our description of $\X_{[p,q]}$ and the form of $\OO_{[p,q]}$ 
given in \eqref{simple form of fibres}.

Ad (ii). Implication \eqref{freeness imply aperiodicity} is straightforward.  For the converse, let us  additionally  assume that $P$ is right cancellative and right reversible (then $P$ is both left and right Ore).  Take any $g_1$,..., $g_{n}\in G(P)\setminus \{[e,e]\}$. Using left reversibility of $P$ we may represent these elements in the form $g_1=[t,r_1]$,..., $g_{n}=[t,r_n]$, where $t, r_1,...,r_n\in P$ and $t\neq r_i$ for $i=1,...,n$. By right reversibility of $P$,  one can inductively  find elements $q_1,...,q_n, p_1',p_2',...,p_n'\in P$ such that 
\begin{align*}
q_1t &=p_1'r_1, \\
q_2q_1t&= p_2'p_1'r_2,\\
&\, ...\\
q_n... q_2 q_1 t&= p_n' ... p_2'p_1'r_n.
\end{align*}
Then defining  
$$
q:=q_n...q_1, \qquad  s:=qt  \quad\textrm{ and }\quad  p_i:=q_n...q_{i+1}p_i' ... p_1' \quad\textrm{ for }\quad i=1,...,n, 
$$
we get   $s=p_ir_i$ and $p_i\neq q$ for $i=1,...,n$. Hence  $q^{-1}s=t$ and  $p_i^{-1}s=r_i$ 
for every $i=1,...,n$. Thus
$$
\X_{g_i}=\X_{[t,r_i]}=\X_t \X_{r_i}^{-1}=\X_{q^{-1}s} \X_{p_i^{-1}s}^{-1}.
$$
Since $\X_{q^{-1}s} \X_{p_i^{-1}s}^{-1}=\X_{[q^{-1}s,p_i^{-1}s]}$ 
does not depend on the choice of $s \geq q, p_i$, we see that the aperiodicity condition applied to $q$ and $p_1,...,p_n$ yields the topological freeness condition for $g_1,...,g_n$.
\end{proof}

We do not know if the converse to implication \eqref{freeness imply aperiodicity} holds  in general, see also Remark \ref{Remark on topological freeness and aperiodicity} below.  Nevertheless, applying Proposition 
\ref{partial dynamical system out of the product system} and Theorems \ref{Cuntz-Krieger uniqueness theorem} 
and \ref{simplicity}, we obtain the following.

\begin{cor} \label{uniqueness theorem and simplicity criterion for cross-sectional algebras}
Suppose  $\{B_{g}\}_{g\in G}$ is a saturated Fell bundle. Treating its fibers as imprimitivity 
Hilbert bimodules over $B_e$, cf. Remark \ref{saturated bundels}, the dual semigroup 
$\{\widehat{B}_{g}\}_{g\in G}$ is a group of genuine homeomorphisms of $\widehat{B}_e$.
\begin{itemize}
\item[i)] The action $\{\widehat{B}_{g}\}_{g\in G}$ is topologically free if and only if the  product system $X=\bigsqcup_{g\in G} B_g$ is topologically aperiodic.  If this is the case, then every $C^*$-norm on  $\bigoplus_{g\in G}B_g$ is topologically graded.
\item[ii)] If the action $\{\widehat{B}_{g}\}_{g\in G}$ is topologically free and has no invariant non-trivial 
open subsets then the reduced cross-sectional $C^*$-algebra $C^*_r(\{B_{g}\}_{g\in G})$ is simple.
\end{itemize} 
\end{cor}


\subsection{Crossed products of twisted $C^*$-dynamical systems} 
 
Suppose $\alpha$ is an action of a semigroup $P$ by endomorphisms of $A$ such that  each $\alpha_s$, $s\in P$, extends to a strictly continuous endomorphism $\overline{\alpha}_s$ of the multiplier algebra $M(A)$. 
Let $\omega$ be a circle-valued multiplier on $P$. That is $\omega:P\times P\to  \T$ is such that 
 $$
 \omega(p,q)\omega(pq,r)=\omega(p,qr)\omega(q,r), \qquad p,q,r \in P.
 $$
Then $(A,\alpha, P ,\omega)$ is called a \emph{twisted semigroup $C^*$-dynamical system}. 
A \emph{twisted crossed product} $A \times_{\alpha,\omega} P$, see \cite[Definition 3.1]{F99},  is the universal $C^*$-algebra generated by $\{i_A(a)i_P(s):a\in A, s\in P\}$, where  $(i_A, i_P)$ is a universal  covariant representation of $(A,P,\alpha,\omega)$. That is,  
$i_A :A \to A \times_{\alpha,\omega} P$  is a homomorphism and  $\{ i_P(p) : p \in
P \}$ are  isometries in $M(A \times_{\alpha,\omega} P)$ such that 
$$
i_P(p) i_P(q) = \omega(p,q) i_P(pq) \quad\text{ and }\quad i_P(p) i_A(a) i_P(p)^* = i_A(\alpha_p(a)),
$$
for $p,q \in P$ and $a \in A$. A necessary condition for $i_A$ to be injective is that all endomorphisms $\al_p$, $p\in P$, are injective. We  apply Theorem \ref{structure theorem}  to show that  when $P$ is of Ore type this condition is also sufficient. Additionally, we reveal a natural Fell bundle structure in $A \times_{\alpha,\omega} P$.

Following \cite{F99}, we associate to  $(A,\alpha, P ,\omega)$  a  product system $X=\bigsqcup_{p\in P^{op}} X_p$ over the opposite semigroup $P^{op}$. We equip the linear space $X_p:=\alpha_p(A)A$ with the following $C^*$-correspondence operations
$$
a\cdot  x= \alpha_{p}(a)x, \qquad x\cdot a =x a, \quad \langle  x, y \rangle_{p}=x^*y,
$$
 $a\in A$, $x,y\in X_p$. The multiplication in $X$ is defined  by
 $$
x \cdot y = \overline{\omega(q,p)}\alpha_q(x)y,\quad  \textrm{ for }\,\,x \in X_p=
\alpha_p(A)A\,\,\textrm{ and }\,\,y\in X_q=\alpha_q(A)A .
$$
By \cite[Lemma 3.2]{F99},  $X$ is  a product system and the left action of $A$  on each of its  fibers is  by compacts. Accordingly, $X$ is a regular product system if and only if all the endomorphisms $\al_p$, $p\in P$, are injective.  Moreover, by \cite[Proposition 3.4]{F99} there is an isomorphism 
$$A\rtimes_{\alpha,\omega} P \cong \OO_X$$
given by the mapping  that sends an element $i_P(p)^*i_A(a) \in A\rtimes_{\alpha,\omega} P$ to the image  of the element $a\in X_p=\alpha_p(A)A$ in $\OO_X$.  Using this isomorphism and Theorem \ref{structure theorem} one immediately gets the following. 

\begin{prop}\label{pomocnicze on crossed products}
Suppose  that  $(A,\alpha, P ,\omega)$ is a twisted semigroup $C^*$-dynamical system, where $P$ is of 
Ore type and all the endomorphisms $\al_p$, $p\in P$, are injective. Then the following hold. 
\begin{itemize}
\item[i)] The algebra $A$ embeds via $i_A$ into the crossed product $A\rtimes_{\alpha,\omega} P$.
\item[ii)] The crossed product $A\rtimes_{\alpha,\omega} P$ is naturally graded over the group of fractions $G(P)$ by  the subspaces of the form 
$$
B_g:=\clsp\{i_P(p)^* i_A(a)i_P(q) : a\in \alpha_p(A)A\alpha_q(A),\,\, [p,q]=g\}, \qquad g\in G(P).
$$
Moreover, $A\rtimes_{\alpha,\omega} P$ can be identified with the cross-sectional $C^*$-algebra $C^*(\{B_{g}\}_{g\in G(P)})$. 
\end{itemize}
\end{prop} 

In the remainder of this subsection we keep the assumptions of Proposition \ref{pomocnicze on crossed products}. 
It is natural to define  a \emph{reduced twisted crossed product} $A \times^r_{\alpha,\omega} P$ to be the
reduced cross-sectional algebra of the Fell bundle $\{B_{g}\}_{g\in G(P)}$. Let $\lambda: A \times_{\alpha,\omega} P \to A \times^r_{\alpha,\omega} P$ be the canonical epimorphism, and 
$$
I_\lambda:=\ker \lambda. 
$$
We wish to generalize the main results of \cite{Arch_Spiel}  to the case of twisted semigroup actions.  
Let $X$ be a product system associated to $(A,\alpha, P ,\omega)$ as above. 
 One can see, cf., for instance, \cite[Example 1.12]{kwa-doplicher}, that a fiber $X_p$, $p\in P$, is a Hilbert bimodule if and only if the range of $\alpha_p$ is a hereditary subalgebra of $A$. If this is the case, then $\alpha_p(A)$ is a corner in $A$:
 $$
\alpha_p(A)=\alpha_p(A)A\alpha_p(A)=\overline{\alpha}_p(1)A \overline{\alpha}_p(1),
$$ and the left inner product in $X_p$ is defined by
$$
{_{p}\langle}  x, y \rangle=\alpha_p^{-1}\left(xy^*\right), \qquad x,y\in X_p=\alpha_p(A)A.
$$ 
The spectrum of $\alpha_p(A)$ can be identified with an open subset  of $\SA$. Then the homeomorphism $\widehat{\alpha}_p: \widehat{\alpha_p(A)}\to \widehat{A}$ dual to  the isomorphism $\alpha_p:A\to\alpha_p(A)$ can be naturally treated as a partial homeomorphism of 
$\SA$, cf. \cite[Definition 2.16]{kwa-interact}. The following Lemma \ref{lemma 5.6} is based on 
\cite[Proposition 2.18]{kwa-interact} dealing with interactions on unital algebras. 

\begin{lem}\label{lemma 5.6} If the monomorphism  $\alpha_p$ has a hereditary range, then the 
homeomorphisms $\widehat{\alpha}_p: \widehat{\alpha_p(A)}\to \widehat{A}$ 
and $\X_p:\widehat{\langle X_p, X_p\rangle_p}\to \widehat{A}$ coincide.
\end{lem}
\begin{proof}
With our identifications, we have  
$$
\widehat{\alpha_p(A)}=\{[\pi]\in \SA: \pi(\alpha_p(A))\neq 0\}=\widehat{\langle X_p, X_p\rangle_p}.
$$
Let $\pi:A\to \B(H)$ be   an irreducible representation such that  $\pi(\alpha_p(A))\neq 0$.  Then $\widehat{\alpha}_p([\pi])$ is the equivalence class of the representation $\pi\circ\alpha_p:A\to \B(\pi(\alpha_p(A))H)$. Since $\pi(\alpha_p(A))H=\pi(\alpha_p(A)A)H$ and 
$$
\|\sum_{i} a_i \otimes_\pi h_i\|^2=\|\sum_{i,j} \langle h_i, \pi( a_i^*a_j)h_j\rangle_p\|= \|\sum_{i} \pi(a_i)h_i\|^2,
$$
  $a_i \in X_p=\alpha_p(A)A$, $h_i\in H$,  $i=1,...,n$, we see  that $a\otimes_\pi h \mapsto \pi(a)h $  yields a unitary operator $U:X_p\otimes_{\pi} H \to \pi(\alpha_p(A))H$. Furthermore, for $a\in A$, $b\in \alpha_p(A)$ and $h\in H$ we have
$$
 [X_p\dashind(\pi)(a) U^*] \pi(b)h= X_p\dashind(\pi)(a)\,  b \otimes_\pi h= (\alpha_p(a)b)\otimes_\pi h=  [U^* (\pi\circ \alpha_p)(a)] \pi(b)h.
$$
Hence $U$  intertwines $X_p\dashind$ and  $\pi\circ\alpha_p$. This proves that  $\X_p=\widehat{\alpha}_p$.
\end{proof}

Before stating our criterion of simplicity for semigroup crossed products, we need to define minimality for 
semigroup actions. 

\begin{defn}\label{minimalityforsemigroups}
Let $\alpha$ be an action of a semigroup $P$ on a $C^*$-algebra $A$. We say that $\alpha$ is {\em minimal} 
if for every ideal $J$ in $A$ such that $\alpha_p^{-1}(J)=J$ for  all $p\in P$ 
we have  $J=A$  or $J=\{0\}$. 
\end{defn}

Let us note that if $X$ is the product system associated to a twisted semigroup $C^*$-dynamical system 
$(A,\alpha, P ,\omega)$ then minimality of $\alpha$ in the sense of Definition \ref{minimalityforsemigroups} 
is equivalent to minimality of $X$ in the sense of Definition \ref{minimalsystem}. 

\begin{prop}\label{partial dynamical system out of semigroup of endomorphisms}
Suppose $(A,\alpha, P ,\omega)$ is  a twisted semigroup $C^*$-dynamical system with  $P$ of Ore type. 
We assume that each endomorphism $\alpha_p$, $p\in P$, is injective and has hereditary range.  As above, we regard $\widehat{\alpha}_p$, $p\in P$, as partial homeomorphisms of $\SA$. The formulas 
$$
D_{[q,p]}:=\widehat{\alpha}_q(\widehat{\alpha_p(A)}), \qquad 
\widehat{\alpha}_{[p,q]}([\pi]):=\widehat{\alpha}_p (\widehat{\alpha}_q^{-1}([\pi])),\qquad [\pi]\in D_{[q,p]}, \,\, p,q\in P,
$$
yield a well defined partial action $(\{D_g\}_{g\in G(P)}, \{\widehat{\alpha}_g\}_{g\in G(P)})$ which coincides with the partial action induced by the Fell bundle $\{B_{g}\}_{g\in G(P)}$ described in Proposition \ref{pomocnicze on crossed products} part (ii). Moreover,
$$
(\{D_g\}_{g\in G(P)}, \{\widehat{\alpha}_g\}_{g\in G(P)}) \textrm{ is topologically free } \Longrightarrow\,\, \{\widehat{\alpha}_p\}_{p\in P}\textrm{ is topologically aperiodic},
$$
and 
\begin{itemize}
\item[i)] if  the semigroup  $\{\widehat{\alpha}_p\}_{p\in P}$ is topologically aperiodic, then for any ideal $I$ in $A \times_{\alpha,\omega} P$ such that $I\cap A=\{0\}$ we have $I\subseteq I_{\lambda}$;
\item[ii)] if the semigroup $\{\widehat{\alpha}_p\}_{p\in P}$  is topologically aperiodic and $\alpha$ is 
minimal,  then the reduced twisted crossed product $A\rtimes_{\alpha,\omega}^r P$  is simple.
\end{itemize} 
\end{prop}
\begin{proof} With the identification of $A\rtimes_{\alpha,\omega} P$ with $\OO_X$, for each $g\in G(P)$
we have the correspondence between $(\OO_X)_g$ and $B_g$. Thus  Lemma \ref{lemma 5.6} and Proposition \ref{partial dynamical system out of the product system} imply the initial part of the assertion. The 
remaining claims (i) and (ii) follow from Lemma \ref{lemma 5.6} and Theorems \ref{Cuntz-Krieger uniqueness theorem} and \ref{simplicity}.
\end{proof}

\begin{rem}\label{Remark on topological freeness and aperiodicity} 
If $P=G$ is a group,  the multiplier $\omega\equiv 1$ is trivial, and all $\alpha_p$, $p\in P$, are automorphisms, then $A \times_{\alpha,\omega} P=A \times_{\alpha} G$ is the classical crossed product. Then   
parts (i) and (ii) of Proposition \ref{partial dynamical system out of semigroup of endomorphisms} coincide with 
\cite[Theorem 1]{Arch_Spiel} and \cite[Corollary on p. 122]{Arch_Spiel}, respectively. 
More generally, let us suppose that $\omega$ is arbitrary,  $P$ is left Ore semigroup, 
and  $\alpha:P\to \Aut (A)$ is a semigroup action by automorphisms. 
By \cite[Theorem 2.1.1]{Laca}  both the action $\alpha$ and the multiplier $\omega$ extend uniquely to 
the group $G=PP^{-1}$  in such a way that $(A,\alpha, G ,\omega)$ is a twisted group $C^*$-dynamical 
system and  we have a natural isomorphism 
$$
A \times_{\alpha,\omega} P \cong A \times_{\alpha,\omega} G.
$$
Then the partial action of $G$ described in Proposition \ref{partial dynamical system out of semigroup of endomorphisms} is by homeomorphisms and coincides with the standard action $\widehat{\alpha}$ 
of $G$ on $\widehat{A}$. Now, when $P$ is both left and right Ore we can infer from  Proposition 
\ref{partial dynamical system out of the product system} part (ii) that  
$$
\textrm{semigroup  }\{\widehat{\alpha}_p\}_{p\in P}\textrm{ is topologically aperiodic }\,\, \Longleftrightarrow\,\, \textrm{ group  }\{\widehat{\alpha}_g\}_{g\in G} \textrm{ is topologically free.}
$$  
It also follows from  \cite[Theorem 2]{Arch_Spiel} and the implication in Proposition 
\ref{partial dynamical system out of semigroup of endomorphisms} 
 that  the above equivalence holds when $P$ is an arbitrary left Ore semigroup and  $A$ is commutative. 
In fact, in this case both these conditions are equivalent to the intersection property described in  Proposition \ref{partial dynamical system out of semigroup of endomorphisms} part (i). 
\end{rem}


\subsection{Topological graph algebras} 

Let  $E=(E^0,E^1,s,r)$ be a \emph{topological graph} as introduced in \cite{ka1}. This means  we assume that  vertex set $E^0$ and edge set $E^1$ are locally compact Hausdorff spaces,  source map  
$s:E^1\to E^0$ is a local homeomorphism, and range map $r:E^1\to E_0$ is a continuous map. 

A \emph{$C^*$-correspondence $X_E$ of the topological graph }$E$ is  defined in the following manner,  \cite{ka1}. The space
$X_E$ consists of  functions $x\in C_0(E^1)$ for which 
$$
 E^0 \ni v \longmapsto \sum_{\{e\in E^1: s(e) = v\}} |x(e)|^2
$$
belongs to $A := C_0(E^0)$. Then $X_E$ is a $C^*$-correspondence over $A$ with the following structure. 
\begin{align*}
(x\cdot a)(e) &:= x(e)a(s(e)) \ \mbox{ for } e\in E^1,\\
\langle x,y\rangle_A (v) &:= \sum_{\{e\in E^1: s(e) = v\}}
\overline{x(e)}y(e) \ \mbox{ for }v\in E^0, \mbox{ and}\\
(a\cdot x)(e) &:= a(r(e))x(e) \ \mbox{ for } e\in E^1.
\end{align*}
$C^*$-correspondence $X_E$ generates a product system over $\N$. It follows from 
\cite[Proposition 1.24]{ka1} that this product system (or simply, this $C^*$-correspondence $X_E$) is 
regular if and only if 
\begin{equation}\label{equivalents of standing assumptions for graphs}
\overline{r(E^1)}= E^0 \textrm{ and every }v\in E^0 \textrm{ has a neighborhood }V\textrm{ such that }r^{-1}(V) \textrm{ is compact}.
\end{equation}
In particular, \eqref{equivalents of standing assumptions for graphs} holds whenever $r:E^1\to E^0$ is a proper surjection. If both $E^0$ and $E^1$ are discrete then $E$ is just a usual directed graph and then  \eqref{equivalents of standing assumptions for graphs} says that every vertex in $E^0$ receives at least one  and at most finitely many edges (in other words, graph $E$ is row-finite and without sources). 
According to \cite[Definition 2.10]{ka1}, the $C^*$-algebra of $E$ is 
$$
C^*(E):= \OO_{X_E}.$$
Let $e = (e_n, . . ., e_1)$, $r(e_i)=s(e_{i+1})$, $i=1,...,n-1$, be a path in $E$. Then $e$ is  a cycle 
if $r(e_n) = s(e_1)$, and vertex $s(e_1)$  is called the base point of $e$. A cycle $e$  is said to be without 
entries if $r^{-1}(r(e_k)) = e_k$ for all $k = 1, . . ., n$. Graph $E$ is \emph{topologically free}, 
\cite[Definition 5.4]{ka1}, if base points of all cycles without entries in $E$ have empty interiors.   
It is known, see \cite[Theorem 6.14]{ka4}, that topological freeness of $E$ is  equivalent to the 
uniqueness property for $C^*(E)$. 

In general, topological aperiodicity of $X_E$ is stronger than topological freeness of $E$. 
However, when $E=(E^0,E^0,s, id)$ is a graph that comes from a mapping $s:E^0\to E^0$, these two 
notions coincide. 

\begin{prop}\label{topological graphs aperiodicity}
Suppose $X_E$ is a $C^*$-correspondence of a topological graph $E$ satisfying \eqref{equivalents of standing assumptions for graphs}. The dual $C^*$-correspondence acts on $E^0$ (identified with the spectrum of $A=C_0(E^0)$) via the formula 
\begin{equation}\label{dual to graph correspondence}
\X_E(v)=r(s^{-1}(v)).
\end{equation}
In particular,
\begin{itemize}
\item[i)]  $X_E$ is topologically aperiodic if and only if the set of base points for periodic paths in  
$E$ has empty interior; 
\item[ii)] If $r$ is injective, topological aperiodicity of $X_E$ is equivalent to topological freeness of $E$;
\item[iii)] If $E$ is discrete, then $X_E$ is topologically aperiodic if and only if $E$ has no cycles, and this  
in turn  is equivalent (see \cite[Theorem 2.4]{kum-pask-rae}) to  $C^*(E)$ being an AF-algebra.
\end{itemize}
\end{prop}
\begin{proof} 
 We identify $\SA$ with $E^0$ by putting $v(a):=a(v)$ for  $v\in E^0$, $a\in A=C_0(E^0)$.  We fix  $v\in E^0$ and an orthonormal basis $\{x_{e}\}_{e\in s^{-1}(v)}$ in  the  Hilbert space $\C^{|s^{-1}(v)|}$.  Let us consider the representation $\pi_v:A\to \B(\C^{|s^{-1}(v)|})$  given by
$$
\pi_v(a)=\sum_{e\in s^{-1}(v)} a(r(e))x_e, \qquad a \in A=C_0(E^0). 
$$ 
One readily checks that the mapping 
$$
X_E\otimes_v \C \ni x\otimes_{v} \lambda \longmapsto \sum_{e\in s^{-1}(v)} \lambda x(e) x_e \in \C^{|s^{-1}(v)|}
$$ 
gives rise to a unitary which establishes    equivalence 
$X_E\dashind (v) \cong \pi_v .
$ 
Furthermore, we have
$$
\{w\in E^{0}: w \leq \pi_v\}=\{w\in E^{0}: w=r(e) \textrm{ for some } e\in s^{-1}(v)\}=r(s^{-1}(v)).
$$
This yields \eqref{dual to graph correspondence}. Claim (i) follows  from \eqref{dual to graph correspondence}, 
part (iii) of Proposition \ref{special cases proposition}  and  the Baire category theorem. Claims (ii) and (iii) are now straightforward.
\end{proof}

\begin{cor}\label{corollary on minimality}
Keeping the assumptions of Proposition \ref{topological graphs aperiodicity}, let $V\subseteq E^0$ be closed. 
Then ideal $J=C_0(E^0\setminus V )$  is   $X_E$-invariant if and only if $\X_E(V)=V$.
\end{cor}
\begin{proof}
It is known, see for instance \cite[Section 2]{ka4}, that ideal $J=C_0(E^0\setminus V )$  is   
$X_E$-invariant if and only if $V$ satisfies the following two conditions 
$$
1) \,\,\,(\forall{e\in E^1}) \,\, s(e) \in V \Longrightarrow r(e) \in V,\,\,\quad \textrm{ and }
\quad\,\,2)\,\,\, v\in V  \Longrightarrow (\exists{e\in r^{-1}(v)}) \,\, s(e)\in V.
$$
In view of \eqref{dual to graph correspondence}, conditions (1) and (2) are respectively  equivalent to 
the inclusions $\X_E(V)\subseteq V$ and  $V\subseteq \X_E(V)$.
\end{proof}


\begin{ex}[Exel's crossed product for a proper local homeomorphism]\label{Exel-Vershik}
  
Let $A=C_0(M)$ for a locally compact Hausdorff space $M$ and let $\al:A\to A$ be the operator of composition with a proper surjective local homeomorphism $\sigma:M\to M$. Then $\al$ is an extendible monomorphism possessing a natural left inverse transfer operator   $L:A\to A$, defined by 
$$
L( a)(t)=\frac{1}{|\sigma^{-1}(t)|}\sum_{s\in\sigma^{-1}(t)}a(s),
$$
see \cite[Subsection 2.1]{brv}. Let $X_L$ be the $C^*$-correspondence with coefficients in $A$, 
constructed as follows. $X_L$ is the completion of $A$ with respect to the norm given by the inner-product below, and with the following structure: 
$$
x\cdot a=x\alpha(a),\qquad  \langle x,y\rangle=L( x^*y),\qquad a\cdot x=ax ,
$$
where $a\in A$, $x,y\in X_L$. Clearly, the left action of $A$ on $X_L$ is injective. One can also show that it is 
by compacts, see the argument preceding \cite[Corollary 4.2]{brv}. Hence $X_L$ is a regular 
$C^*$-correspondence. It is known that is naturally isomorphic to a $C^*$-correspondence associated to the topological graph $E=(M,M,\sigma,id)$, \cite[Section 6]{brv}. Thus, by Proposition \ref{topological graphs aperiodicity}, the dual $C^*$-correspondence to $X_L$  acts on $M$, identified with the spectrum of 
$A=C_0(M)$, via the formula 
\begin{equation}\label{dual to graph correspondence2}
\X_L(t)=\sigma^{-1}(t). 
\end{equation}

It is observed in  \cite{brv} that 
$$
C_0(M)\rtimes_{\alpha,L}\N:=\OO_{X_L}
$$ 
is a natural candidate for Exel's crossed product when $A=C_0(M)$ is non-unital.  
When $M$ is compact, $C(M)\rtimes_{\alpha,L}\N$ coincides with the  crossed product introduced in \cite{exel2} and can be effectively described  in terms of generators and relations, \cite[Theorem 9.2]{exel_vershik}. 

Now, combining Proposition \ref{topological graphs aperiodicity}, \cite[Lemma 6.2]{brv} and 
\cite[Theorem 6.14]{ka4}, we see that the following conditions are equivalent. 
\begin{itemize}
\item[i)] $X_L$ is topologically aperiodic; 
\item[ii)] the set of periodic points of $\sigma$ has empty interior;
\item[iii)] $\sigma$ is topologically free in the sense of Exel and Vershik \cite[Definition 10.1]{exel_vershik},  \cite{brv};
\item[iv)] every non-trivial ideal in $C_0(M)\rtimes_{\alpha,L}\N$ intersects $C_0(M)$ non-trivially.
\end{itemize}
Consequently, in view of Corollary \ref{corollary on minimality}, the crossed product $C_0(M)\rtimes_{\alpha,L}\N$ is simple if and only if in addition to the above equivalent conditions 
there is no nontrivial closed subset $Y$ of $M$ such that $\sigma^{-1}(Y)=Y$, cf. \cite[Theorem 6.4]{brv}, \cite[Theorem 11.2]{exel_vershik}, see also \cite{CarSilve} and \cite{Stam}.
\end{ex}


\subsection{$C^*$-algebras of topological $P$-graphs}
\label{product system of topological graphs over P}
In this subsection, we  introduce topological $P$-graphs which generalize both  topological $k$-graphs \cite{yeend} and (discrete) $P$-graphs \cite{Raeburn_Sims}, \cite{BSV}.
Within the framework of a general approach to product systems proposed in \cite{fs},  
the reasoning in \cite[Example 1.5 (4)]{fs} shows that a topological $P$-graph defined below is simply  
a product system over $P$ with values in a groupoid of topological graphs, see  \cite[Definition 1.1]{fs}.
In the sequel $P$ is a  semigroup of Ore type. We treat elements of $P$ as morphisms in 
a category with single object $e$. 
\begin{defn} 
By a \emph{topological $P$-graph} we mean a pair $( \Lambda, d )$ consisting of: 
\begin{itemize}
\item[(1)] a small category $\Lambda$ endowed with a second countable
locally compact Hausdorff topology under which the composition map
is continuous and open, the range map $r$ is continuous and the source
map $s$ is a local homeomorphism; 
\item[(2)] a continuous functor $d \colon \Lambda\to P$, called degree map, satisfying the
factorization property: if $d(\lambda) = pq$ then there exist
unique $\mu,\nu$ with $d(\mu) = p$, $d(\nu) =q $ and $\lambda = \mu\nu$.
\end{itemize}
\end{defn}
Elements (morphisms) of $\Lambda$ are called paths.  $\Lambda^{p}:= d^{-1} (p)$ stands for the 
set of paths of degree $p \in P$.  Paths of degree $e$ are called vertices. 

We associate to a topological $P$-graph  $(\Lambda, d )$ a product system in the same manner as it is done for topological $k$-rank graphs in \cite{CLSV}. That is, for each $p \in P$ we let $X_{p}=X_{E_p}$ be the 
standard $C^*$-correspondence associated to the topological graph 
$$
E_p=( \Lambda^{e}, \Lambda^{p}, s |_{\Lambda^{p}}, r |_{\Lambda^{p}} ),
$$
so that $A := C_{0} ( \Lambda^{e} )$ and  $X_p$ is the completion of the  pre-Hilbert $A$-module $C_c(\Lambda^p)$  with the structure 
$$
\langle f, g \rangle_{p} ( v ) = \sum_{\eta \in \Lambda^{p}(v)} \overline{f ( \eta )} g (\eta)
\quad\text{ and }\quad ( a \cdot f \cdot b) ( \lambda ) = a (r(\lambda)) f(\lambda) b(s(\lambda)).
$$
The proof of  \cite[Proposition 5.9]{CLSV} works in our more general setting and shows that  the formula 
$$
(fg)(\lambda) := f(\lambda (e,p))g(\lambda (p,pq))
$$
defines a  product  $X_{p}\times X_{q} \ni (f,g) \to fg \in X_{pq}$ that makes 
$X=\bigsqcup_{p\in P} X_p$ into a product system. In view of 
\eqref{equivalents of standing assumptions for graphs}, we see that the product system $X$ is regular 
if and only if for every $p\in P$ we have 
\begin{align*}
 & \overline{r(\Lambda^p)} = \Lambda^0, \textrm{ and } \\
 & \text{every } v\in E^0 \textrm{ has a neighborhood }
V\textrm{ such that }r^{-1}(V)\cap \Lambda^p \textrm{ is compact in }\Lambda^p,
\end{align*}
If the above condition holds, we  say that the \emph{topological $P$-graph $(\Lambda,d)$ is regular}. It follows from \cite[Theorem 5.20]{CLSV}  that if $(\Lambda,d)$ is a regular topological $k$-rank graph (that is, if $P=\N^k$), then the Cuntz-Krieger  algebra of $(\Lambda,d)$ defined in \cite{yeend} coincides with $\OO_X$. Hence it is natural to coin the following definitions, see also Remark \ref{remark on P-graph algebras} below. 

\begin{defn}
Suppose $(\Lambda,d)$ is a regular topological $P$-graph, where $P$ is a semigroup of Ore type.   We define  a \emph{$C^*$-algebra} $C^*(\Lambda,d)$  and a \emph{reduced $C^*$-algebra} $C^*_r(\Lambda,d)$ of $(\Lambda,d)$ to be respectively the Cuntz-Pimsner algebra $\OO_X$  and the reduced Cuntz-Pimsner algebra $\OO_X^r$, where $X$ is the regular product system defined above.
 \end{defn}
\begin{rem}\label{remark on P-graph algebras}
If $\Lambda$ is a discrete space then $C^*(\Lambda,d)$ is a universal $C^*$-algebra generated by partial isometries $\{s_\lambda: \lambda\in \Lambda\}$ subject to a natural version of Cuntz-Krieger relations, see \cite[Theorem 4.2]{Raeburn_Sims}. If we additionally assume $(G,P)$ is a quasi-lattice ordered group then $C^*_r(\Lambda,d)$ coincides with the co-universal $C^*$-algebra $C^*_{min}(\Lambda)$ associated to $(\Lambda,d)$ in \cite{BSV}. To see the latter combine  \cite[Proposition 6.4]{Raeburn_Sims},  \cite[Theorem 5.3]{BSV}, \cite[Theorem 4.1]{CLSV} and  \cite[Corollary 5.2]{SY}.
\end{rem}
As an application of our  main results --- Theorems \ref{structure theorem}, 
\ref{Cuntz-Krieger uniqueness theorem}, \ref{simplicity}, we obtain the following. 

\begin{prop}\label{topological graphs aperiodicity343}
Suppose $(\Lambda,d)$  is a regular topological $P$-graph. The $C^*$-algebras $C^*(\Lambda,d)$  and $C^*_r(\Lambda,d)$ are non-degenerate in the sense that they are generated by the images of injective Cuntz-Pimsner representations of  $X=\bigsqcup_{p\in P} X_p$. Moreover, 
\begin{itemize}
\item[i)] $X$ is topologically aperiodic if and only if  for every nonempty open set $U\subseteq \Lambda^e$, each finite set  $F\subseteq  P$ and an element $q\in P$ with $q\nsim_R p$ for all $p \in F$, there is an   
enumeration $\{p_1,...,p_n\}$ of elements of $F$  and there are elements $s_1,...,s_n\in P$ such 
that $q \leq s_1 \leq  ... \leq s_n$, $p_i\leq s_i$, for $i=1,...,n$, and the union  
\begin{equation}\label{topological freeness condition2}
\bigcup_{i=1}^n\{v\in \Lambda^{e}: \mu \in \Lambda^{p_i^{-1}s_i},  \nu \in \Lambda^{q^{-1}s_i},\,\, s(\mu)=s(\nu)  \textrm{ and } r(\mu)=r(\nu)=v\} 
\end{equation}
 does not contain $U$.
 \item[ii)] $X$ is minimal if and only if  there is no nontrivial closed set $V\subseteq \Lambda^e$ such that 
\begin{equation}\label{invariant sets for P-graphs}
  r(\Lambda^p \cap s^{-1}(V))=V\quad \textrm{ for all } p\in P.
\end{equation}
\end{itemize}
In particular, if the equivalent conditions in (i) hold, then any non-zero ideal in  $C^*_r(\Lambda,d)$  has non-zero intersection with $C_0(\Lambda^e)$.
If the conditions described in (i) and (ii) hold, then  $C^*_r(\Lambda,d)$ is simple.
\end{prop}
\begin{proof}
The initial claim of this proposition follows from Theorem \ref{structure theorem} above. 
To see that the equivalence in part (i) holds, it suffices to apply  formula \eqref{dual to graph correspondence}  
to the $C^*$-correspondences $X_{p}=X_{E_p}$, $p\in P$. Similarly, using \eqref{dual to graph correspondence} and Corollary \ref{corollary on minimality}, we see  that $X$-invariant ideals in $C_0(\Lambda^e)$ are in one-to-one correspondence with closed sets  $V$ satisfying \eqref{invariant sets for P-graphs}. This proves part (ii). 
The final claim of the proposition now follows from Theorems  \ref{Cuntz-Krieger uniqueness theorem} and  \ref{simplicity} above.
\end{proof}
\begin{rem}\label{Yeend condition remark}
Until now, there has been several different aperiodicity conditions introduced that imply uniqueness theorems for topological (or discrete) higher-rank graphs, that is when $P=\N^k$, cf. \cite{Raeburn}, \cite{yeend}, \cite{Yam09}.  To our knowledge there are no such theorems known for more general semigroups $P$.
We also point out that our topological aperiodicity has an advantage  of being {\em local} -- it involves only finite  paths in $\Lambda$, which is of importance, cf. \cite[discussion on page 94]{Raeburn}.  
\end{rem}


\subsection{The Cuntz algebra $\Qq_\N$}\label{Cuntz subsection}

In \cite{Cun1}, Cuntz introduced $\Qq_\N$, the universal $C^*$-algebra generated by a
unitary $u$ and isometries $s_n$, $n\in \N^\times$, subject to the relations
\begin{enumerate}\renewcommand{\theenumi}{Q\arabic{enumi}}
\item $s_ms_n=s_{mn}$,
\item $s_m u=u^m s_m$, and
\item $\sum_{k=0}^{m-1}u^ks_m s_m^*u^{-k}=1,$
\end{enumerate}
for all $m,n\in \N^\times$. Cuntz proved that $\Qq_\N$ is simple and purely infinite. Now we  deduce the simplicity of $\Qq_\N$ from our general result -- Theorem \ref{simplicity} above, see also  Remark \ref{last remark} below.

It was shown in \cite{Yam09} that $\Qq_\N$ may be viewed as the Cuntz-Pimsner algebra 
of a certain product system. We recall  an explicit description of that product system given in \cite{HLS}.

The product system $X$ is over the semigroup $\N^\times$ and its coefficient algebra  is $A=C(S^1)$. 
We denote by $Z$ the standard unitary generator of $A$. 
Each fiber $X_m$, $m\in \N^\times$, is a $C^*$-correspondence over $A$ associated to the classical covering map 
$S^1\ni z \to z^m \in S^1$, as constructed in 
Example \ref{Exel-Vershik}. Each $X_{m}$ as left $A$-module is free with rank $1$, and we denote the 
basis element by $1_m$. Hence, each element of $X_m$ may be uniquely written as 
$\xi 1_m$ with $\xi\in A$. We have 
$$ \begin{aligned}
(\xi 1_m)\cdot a & =\xi \alpha_m(a) 1_m, \\  
\langle \xi 1_m,\eta 1_m\rangle_m & = L_m(\xi^*\eta), \\
a \cdot \xi 1_m & = (a\xi) 1_m, 
\end{aligned} $$
for $\xi, a\in A$. Then 
$$
X := \bigsqcup_{m\in\N^\times}X_m
$$
becomes  a product system with multiplication $X_m\times X_r\to X_{mr}$ given by
$$
(\xi 1_m)(\eta 1_r):=(\xi\alpha_m(\eta)) 1_{mr}
$$
for $m,r\in\N^\times$. By \cite[Proposition 3.13]{HLS} (cf. \cite[Corollary 5.2]{SY}) we have
 $$
 \OO_X\cong \Qq_\N.
 $$ 
Now, let $E_{i,j}$, $i,j=0,1,\ldots,m-1$, be a system of matrix units in $M_m(\C)$. There is an 
isomorphism 
$$ C(S^1) \otimes M_m(\C) \cong \K(X_m) $$ 
such that
$$ f \otimes E_{i,j} \leftrightarrow  \Theta_{Z^i\alpha_m(f)1_m,Z^j 1_m}. $$
Thus $\widehat{\K(X_m)}$ may be identified with the circle $S^1$. With these identifications, we have  
$$ \phi_m(Z) = Z\otimes E_{0,m-1} + \sum_{j=0}^{m-2} 1\otimes E_{j+1,j}, $$
and hence the multivalued map $\widehat{\phi_m}:S^1\to S^1$ is such that
$$ \widehat{\phi_m}(z) = \{w\in S^1 \mid w^m=z\}. $$
Furthermore, $[X_m\dashind]$ is identified with the identity map on $S^1$, and consequently 
the multivalued map $\widehat{X_m}=\widehat{\phi_m}\circ[X_m\dashind]:S^1\to S^1$ is 
$$ \widehat{X_m}(z)= \{w\in S^1 \mid w^m=z\}. $$
For $m\neq n$ the set $\{z\in Z^1 \mid z\in \widehat{X_m}(\widehat{X_n}^{-1}(z))\}$ is 
finite, while every nonempty open subset of $S^1$ is infinite. It follows that the product system 
$X$ is topologically aperiodic. 

Now, we see that $A$ does not contain any non-trivial invariant ideals. 
Indeed, suppose $J$ is an $X$-invariant ideal in $A$. Then 
$L_m(J)\subseteq J$ for all $m\in\N^\times$.  There exists an open subset $U$ of $S^1$ 
and a function $f\in J$ such that 
$f\geq 0$ and $f(t)\neq 0$ for all $t\in U$. If $m$ is sufficiently large then for each $z\in S^1$ 
there is a $w\in U$ such that $w^m=z$. Then $L_m(f)$ is strictly positive on $S^1$ and hence invertible. 
Since $L_m(f)\in J$, we conclude that $J=A$. 

\begin{rem}\label{last remark}
We recall, cf.  \cite[Section 2]{brv} and Example \ref{Exel-Vershik}, that for each $m\in \N^\times$ the mapping 
$$
 C(S^1)\ni a \mapsto\sqrt{m} a 1_m\in X_m
$$ establishes isomorphism between  the fiber $X_m$  and the  $C^*$-correspondence associated to the  topological graph $(S^1,S^1,\alpha_m,id)$. Using these  isomorphisms, one may recover the product system associated to the topological $\N^\times$-graph $(\Lambda,d)$ constructed in terms of generators in \cite[Proposition 5.1]{Yam09}. In particular, simplicity of $\Qq_\N$ could be also deduced from  Proposition \ref{topological graphs aperiodicity343} applied to $(\Lambda,d)$. Moreover, as the range map in each fiber of $(\Lambda,d)$ is injective, part (ii) of  Proposition \ref{topological graphs aperiodicity}  and Example  
\ref{Exel-Vershik} indicate that our simplicity criterion in this case might be not only sufficient but also necessary.
\end{rem}


\end{document}